\documentclass[12pt]{amsart}

\usepackage{amsmath,amsthm,amssymb}
\usepackage{amsfonts}
\usepackage{comment}
\usepackage{amsrefs}
\usepackage{hyperref}
\usepackage{mathtools}
\usepackage{bm}
\usepackage[margin=1in]{geometry}

\usepackage{mathrsfs}
\usepackage{enumerate}

\usepackage{todonotes}

\newtheorem{theorem}{Theorem}[section]
\newtheorem{proposition}[theorem]{Proposition}
\newtheorem{lemma}[theorem]{Lemma}
\newtheorem{corollary}[theorem]{Corollary}

\theoremstyle{definition}

\newtheorem*{theorem*}{Theorem}
\newtheorem*{proposition*}{Proposition}
\newtheorem*{lemma*}{Lemma}

\theoremstyle{remark}
\newtheorem*{remark}{Remark}

\numberwithin{equation}{section}

\newcommand\abs[1]{\left|#1\right|}

\newcommand\ceil[1]{\left\lceil#1\right\rceil}
\newcommand\floor[1]{\left\lfloor#1\right\rfloor}

\newcommand{\QQ}{\mathbb Q}

\newcommand{\NN}{\mathbb N}
\newcommand{\RR}{\mathbb R}

\newcommand{\on}{\operatorname}
\newcommand{\rank}{\on{rank}}
\newcommand{\Cl}{\on{Cl}}

\makeatletter
\newcommand*{\defeq}{\mathrel{\rlap{%
                     \raisebox{0.3ex}{$\m@th\cdot$}}%
                     \raisebox{-0.3ex}{$\m@th\cdot$}}%
                     =}
\makeatother

\newcommand{\eps}{\varepsilon}
\newcommand\half{\frac12}

\newcommand{\PP}{\mathbb P}

\usepackage{mathrsfs}

\author[Yang Liu]{Yang Liu}
\address{Department of Mathematics, Massachusetts Institute of Technology, \mbox{Cambridge, MA 02139}}
\email{\href{mailto:yliu@mit.edu}{{\tt yliu97@mit.edu}}}

\author[Peter S. Park]{Peter S. Park}
\address{Department of Mathematics, Princeton University, \mbox{Princeton, NJ 08544}}
\email{\href{mailto:pspark@math.princeton.edu}{{\tt pspark@math.princeton.edu}}}

\author[Zhuo Qun Song]{Zhuo Qun Song}
\address{Department of Mathematics, Princeton University, \mbox{Princeton, NJ 08544}}
\email{\href{mailto:zsong@princeton.edu}{{\tt zsong@princeton.edu}}}

\begin{document}

\title[Bounded Gaps Between Products of Distinct Primes]{Bounded Gaps Between Products of Distinct Primes}
\date{\today}

\begin{abstract}
Let $r \ge 2$ be an integer. We adapt the Maynard--Tao sieve to produce the asymptotically best-known bounded gaps between products of $r$ distinct primes. Our result applies to positive-density subsets of the primes that satisfy certain equidistribution conditions. This improves on the work of Thorne and Sono.
\end{abstract}

\maketitle

\section{Introduction}\label{sec:introduction}
One of the most famous conjectures in mathematics is the twin prime conjecture, which states that
\[
\liminf_{n \to \infty} (p_{n+1} - p_n) = 2,
\]
where $p_i$ denotes the $i$-th prime. The numerical evidence for this conjecture is striking because the average value of $p_{n+1} - p_n$ is known to grow arbitrarily large; specifically, the prime number theorem implies that the difference is $\sim \log p_n$ on average. One of the first significant advances toward the twin prime conjecture is due to Chen Jingrun \cite{chen}, who proved in 1973 that there are infinitely many primes $p$ such that $p+2$ is a product of at most two primes.

In another direction, Goldston, Graham, Pintz, and Y{\i}ld{\i}r{\i}m  (whom we will refer to as ``GGPY'') considered the problem of proving a finite bound for gaps between primes. In 2005, Goldston, Pintz, and Y{\i}ld{\i}r{\i}m~\cite{gpy} used the Selberg sieve to prove that
\begin{equation}\label{gpyresult}
\liminf_{n \to \infty} \frac{p_{n+1} - p_n}{\log p_n} = 0.
\end{equation}
GGPY~\cite{ggpy} gave an alternative proof of \eqref{gpyresult} that also used a variant of Selberg's sieve weights. These methods rely on the Bombieri--Vinogradov theorem, an equidistribution result stating that 
\begin{equation}\label{uno}
\sum_{q \leq x^\theta} \max_{(a,q) = 1} \abs{\pi(x;q,a) - \frac{1}{\phi(q)}\pi(x)} \ll \frac{x}{(\log x)^A}
\end{equation}
for any $A>0$ and $\theta < \half$, where
\[
\pi(x) \defeq \#\{p \text{ prime} \colon p \leq x\} \hspace{12pt} \text{ and } \hspace{12pt} \pi(x;q,a) \defeq \#\{p \text{ prime} \colon p \leq x, p \equiv a \pmod q\}.
\]
We say that $\theta>0$ is a \emph{level of distribution} if \eqref{uno} holds for any $A>0$. Also, these methods conditionally prove bounded gaps between the primes, i.e, a finite bound for
\[
\liminf_{n \to \infty}(p_{n+1} - p_n),
\]
under the assumption that there exists a level of distribution $\theta > \half$.

The GGPY approach can be applied to gaps between products of primes. Let $\mathbb{P}$ be the set of all primes, and consider a subset $\mathcal{P}\subseteq\mathbb{P}$. Let $E_{r}(\mathcal P)$ denote the set of positive integers that are products of $r$ distinct primes in $\mathcal P$. GGPY \cite{gpylemma} proved that
\[
\liminf_{n \to \infty} (q_{n+1} - q_n) \le 6,
\]
where $q_i$ denotes the $i$-th number in $E_2(\PP)$. Subsequently, Thorne \cite{thorne} extended their method to prove the existence of bounded gaps between any number of consecutive $E_r(\PP)$ numbers for any $r \ge 2$. In fact, Thorne's result extends to $E_r(\mathcal P)$ for subsets $\mathcal P$ of the primes satisfying certain equidistribution conditions, thus yielding a number of intriguing consequences in the context of multiplicative number theory. For instance, for an elliptic curve $E/\QQ$ given by the equation 
\[
E \colon y^2=x^3+ax^2+bx+c
\]
and a fundamental discriminant $D$, let
\[
E(D) \colon Dy^2=x^3+ax^2+bx+c
\]
denote the $D$-quadratic twist of $E$. Thorne proved in \cite[Theorem 1.2]{thorne} that if $E$ has no 2-torsion, then there are bounded gaps between squarefree fundamental discriminants for which $E(D)$ has Mordell--Weil rank $0$ and its Hasse--Weil $L$-function satisfies $L(1,E(D)) \neq 0$. Another corollary, \cite[Corollary 1.3]{thorne}, states that there are bounded gaps between $E_2(\PP)$ numbers $pq$ such that the class group $\Cl(\sqrt{-pq})$ contains an element of order $4$. In 2014, Chung and Li \cite{chungli}
proved an analogous and quantitatively stronger bound for the size of the gaps in the above results for squarefree numbers whose prime divisors are all in $\mathcal{P}$ instead than $E_r(\mathcal{P})$ numbers.

In 2013, Zhang \cite{zhang} unconditionally proved that 
\[
\liminf_{n \to \infty} (p_{n+1} - p_n) < 7 \cdot 10^7
\]
by finding a suitable modification of \eqref{uno} that holds for $\theta > \half$. Zhang's result, while groundbreaking, cannot be easily extended to subsets of primes. Later that year, Maynard \cite{maynard} (and Tao independently) developed a further refinement of the Selberg sieve to show that 
\[
\liminf_{n \to \infty} (p_{n+m} - p_n) \ll m^3 \exp(4m).
\]
Moreover, this refinement produces bounded gaps between primes with \emph{any} level of distribution $\theta > 0$. This suggests that the Maynard--Tao sieve can be adapted to study gaps between primes in special subsets. In particular, Thorner \cite{thorner} extended the methods in \cite{maynard} to show that there exist bounded gaps between primes in Chebotarev sets. There are a number of interesting number-theoretic consequences for ranks of elliptic curves, Fourier coefficients of modular forms, and primes represented by binary quadratic forms.

Sono \cite{sono} used the Maynard--Tao sieve to improve the bound given by \cite{gpylemma} for the gaps between products of two primes. Specifically, he showed that for every $\eps > 0$, 
\[
\liminf_{n \to \infty} (q_{n+m} - q_{n}) \ll \exp\left(\frac{(4+\eps)m}{3\log m}\right)
\]
holds for sufficiently large $m$ in terms of $\eps$. Moreover, Neshime \cite[Remark 7.1]{sono} was able to modify Sono's methods to show an even stronger bound of
\begin{equation}\label{sonostudent}
\liminf_{n \to \infty} (q_{n+m} - q_{n}) \ll \sqrt{m}\exp\left(4\sqrt{m}\right).
\end{equation}

In this paper, we adapt the arguments in \cite{maynard} and \cite{sono} to prove bounded gaps between consecutive $E_r(\mathcal P)$ numbers for any $r \ge 2$ and infinite subset $\mathcal P \subseteq \PP$ satisfying the  equidistribution conditions defined below. Let 
\[
\pi_{\mathcal{P}}(x) \defeq \sum_{ \substack{ p \leq x \\ p \in \mathcal{P}}} 1, \hspace{40pt} \text{and} \hspace{40pt} \pi_{\mathcal{P}}(x;q,a) \defeq \sum_{ \substack{ p \leq x \\ p \in \mathcal{P} \\ p \equiv a \pmod q}} 1.
\]
Let $\delta = \delta (\mathcal{P})$ denote the density of $\mathcal{P}$ in the set of all primes, i.e,
\begin{equation}\label{2.1}
\pi_{\mathcal{P}}(x) = (1 + o(1)) \frac{\delta x}{\log x}.
\end{equation}
For $\mathcal{P}$ such that $\delta (\mathcal{P}) > 0$, we say that $\mathcal{P}$ satisfies a \emph{Siegel--Walfisz condition} if there exists a squarefree positive integer $B = B(\mathcal{P})$ such that for all $(q, B) = 1$ and $(a, q) = 1$,
\begin{equation}
\left |\pi_{\mathcal{P}}(x; q, a) - \frac{1}{\phi(q)}\pi_{\mathcal{P}}(x) \right| \ll \frac{x}{(\log x)^A}\label{2.2}
\end{equation}
holds for any $A>0$.
Finally, for $\mathcal{P}$ that satisfies a Siegel--Walfisz condition, we say that $\mathcal{P}$ has \emph{level of distribution $\theta > 0$} if 
\begin{equation}
\sum_{\substack{q < x^{\theta} \\ (q, B) = 1}} \max_{(a, q) = 1} \left |\pi_{\mathcal{P}}(x; q, a) - \frac{1}{\phi(q)}\pi_{\mathcal{P}}(x) \right | \ll \frac{x}{(\log x)^A}
\label{2.3}
\end{equation}
holds for any $A > 0$, with the implied constant depending on $A$. For the purposes of our paper, assume that $\theta < \frac{1}{2}$ throughout.

Define a set of linear forms $n+h_1, \ldots, n+h_k$ to be \emph{admissible} if for every prime $\ell$, there exists an integer $n_\ell$ such that 
\[
\ell \nmid \prod_{i=1}^k (n_\ell+h_i).
\]
We now state our main theorem in the context of admissible sets of linear forms $n+h_1, \ldots, n+h_k$  that are guaranteed to contain at least $m$ $E_r(\mathcal P)$ numbers, for infinitely many $n$. 

\begin{theorem}\label{thm:sellmain} Suppose that a subset $\mathcal{P}$ of the primes has positive density $\delta = \delta(\mathcal P)$ in the set of all primes, satisfies a Siegel--Walfisz condition as in \eqref{2.2}, and has a positive level of distribution as in \eqref{2.3}, with $B = B(\mathcal P)$ defined accordingly. There exists a constant $C(r,m,\mathcal P)$, depending only on $r, m,$ and $\mathcal P$, such that for any admissible set of linear forms $\{n + h_1, \dots, n + h_k\}$ satisfying $k > C(r, m, \mathcal{P})$, it holds for  infinitely many positive integers $n$ that at least $m$ of $n + h_1, \dots, n + h_k$ are $E_r(\mathcal{P})$ numbers.
\end{theorem}
\noindent Note that the explicit value for the constant $C(r,m,\mathcal P)$ is given in \eqref{accboundforms}.

From Theorem~\ref{thm:sellmain}, we obtain bounded gaps for an arbitrary number of $E_r(\mathcal P)$ numbers.

\begin{corollary}\label{col:show} Let $a_{i}$ denote the $i$-th $E_r(\mathcal P)$ number. Then, there exists a constant $C(r,m,\mathcal P)$, depending only on $r, m,$ and $\mathcal P$, such that 
\[
\liminf_{n \to \infty} (a_{n+m} - a_{n}) < C(r,m,\mathcal P).
\]
\end{corollary}
\noindent The explicit value for the constant $C(r,m,\mathcal P)$ of this context is given in \eqref{accbound}.

When $r$ is fixed and $m$ is large, our result improves on the work of \cite{thorne}, thus yielding improved effective constants in the contexts described in \cite[Section 1]{thorne}.

\begin{corollary}\label{corollary0}
Let $K/\QQ$ be a Galois extension, and $r \ge 2$ be an integer. Then, there exists an effectively computable constant $C(r,m,K)$ such that there
are infinitely many nonconjugate $m$-tuples of ideals $\mathfrak a_1,\ldots,\mathfrak a_m$ whose norms are $E_r(\PP)$ numbers that are simultaneously contained in an interval of length $C(r,m,K)$.
\end{corollary}
\begin{proof}
Apply Theorem~\ref{col:show} to the Chebotarev set of primes that are inert in $K$.
\end{proof}
\begin{corollary}\label{corollary1}
Let $E/\QQ$ be an elliptic curve without $2$-torsion. Then, there exists an effectively computable squarefree integer $N_E$ such that the following holds: For every $r\ge 1$, there exists an effective constant $C(r,m,E)$ such that there are infinitely many $m$-tuples of $E_{2r}(\PP)$ numbers $b_1 < \cdots < b_{m}$ for which
\begin{enumerate}
\item $L\left(1,E\left(N_E \cdot b_i\right)\right) \neq 0$ for all $1 \le i \le m$,
\item $\rank E\left(N_E \cdot b_i\right) = 0$ for all $1 \le i \le m$, and 
\item  $b_{m}-b_1< C(r,m,E)$.
\end{enumerate}
\end{corollary}
\begin{proof}
For each $E/\QQ$, one can apply Theorem~\ref{col:show} to a proof analogous to that of \cite[Section 6]{thorne}, which computes the constants for the elliptic curve $X_0(11):y^2 =x^3-4x^2-160x-1264$. The computations rely on the work of Ono \cite{ono}  and the Bombieri--Vinogradov-type result of Murty and Murty \cite{murtymurty}.
\end{proof}
\begin{corollary}\label{corollary2}
Let $r \ge 2$ be an integer. There exists an effectively computable constant $C(r,m)$ such that there are infinitely many $m$-tuples of $E_r(\PP)$ numbers $d_1 < \cdots < d_m$ such that the following hold:
\begin{enumerate}
\item The class group $\Cl\left(\QQ(\sqrt{-d_i})\right)$ has an element of order $4$ for all $1 \le i \le m$.
\item  $d_m-d_1< C(r,m)$.
\end{enumerate}
\end{corollary}
\begin{proof}
Soundararajan~\cite{sound} has proved that for any squarefree $d \equiv 1 \pmod 8$ whose prime factors are congruent to either $1$ or $-1 \pmod 8$, the class group $\Cl\left(\QQ(\sqrt{-d})\right)$ contains an element of order $4$. Our corollary then follows immediately from applying Theorem~\ref{col:show} to $\mathcal P = \{p \in \PP \colon p \equiv 1 \pmod 8\}$. 
\end{proof}
In \cite{thorne}, Thorne proved in the contexts of Corollaries \ref{corollary0}, \ref{corollary1}, and \ref{corollary2}, we can take the constants $C(r, m, K) = C(r, K) \exp(m^{\frac{1}{r-2}})$, $C(r, m, E) = C(r, E) \exp(m^{\frac{1}{r-2}})$, and $C(r, m) = C(r) \exp(m^{\frac{1}{r-2}})$, respectively. When $r$ is fixed and $m$ is large, we improve these bounds to $C(r, m, K) = C(r, K) m^{\frac{1}{r}} \exp(m^{\frac{1}{r}})$, $C(r, m, E) = C(r, E) m^{\frac{1}{r}} \exp(m^{\frac{1}{r}})$, and $C(r, 
m) = C(r) m^{\frac{1}{r}} \exp(m^{\frac{1}{r}})$, respectively. We note that these bounds are not as strong as those obtained by Chung and Li in \cite{chungli} when considering square-free numbers with no restriction on the number of prime factors.

In Section~\ref{sec:preliminaries}, we list the definitions and notations we will be using, state a precise version of our main result, and give an outline of our proof. Then, in Section~\ref{sec:bv}, we prove a version of the Bombieri--Vinogradov theorem for $E_r(\mathcal P)$ numbers. We then use this result to prove our main theorem in the course of Sections~\ref{sec:reduction}, \ref{sec:combinatorial}, \ref{sec:smoothy}, and \ref{sec:smoothf}. Finally, in Section~\ref{sec:example}, we explicitly compute the constant given in Corollary~\ref{corollary2} for $r=2$ and arbitrary $m$, which shows a concrete application of our asymptotically improved result on bounded gaps.

\section{Preliminaries}\label{sec:preliminaries}
We detail below the equidistribution properties that we will assume for our infinite subset $\mathcal{P} \subset \PP$, for which we will show that the gap between $m$ consecutive $E_r(\mathcal P)$ is bounded infinitely often. 
\subsection{Definitions and Notations}

All sums, products, and maxima are taken with the variables ranging over the positive integers $\NN$. The exception to this will be that variables denoted by $p$ or $p_i$ will be assumed to be taken over $\mathcal P$, and the variable $\ell$ will be assumed to be taken over $\PP$.  For the purposes of our paper, the level of distribution  $\theta$ is assumed to be less than $1/2$. Throughout the proof, we work with a fixed admissible set of $k$ distinct linear forms $\mathcal H$, and without loss of generality assume that $h_1 < \cdots < h_k$.

We let $\phi$ denote the Euler's totient function, $\mu$ the M{\"o}bius function, and $\tau_u$ the function given by the number of distinct ways a number can be written as a product of $u$ ordered positive integers. For two positive integers $a,b$, we let $(a,b)$ denote their greatest common divisor and $[a,b]$ their least common multiple. The exception to this is in Section~\ref{sec:bv}, where $[a,b]$ will denote the closed interval with endpoints $a$ and $b$.

We fix $r \geq 2$, and for a given $N$, we further restrict our consideration to products of $r$ distinct primes in $\mathcal P$ such that the prime factors satisfy a size constraint in terms of $N$. Specifically, for every $1 \le h \le r$, let $E_h \defeq \{p_1 \ldots p_h \colon N^\eta \le p_1 < \cdots < p_h \text{ and } N^{\half} \le p_h\}$, where $0 < \eta < \frac{1}{r}$ is fixed. We note that the definition of $E_h$ has an implicit dependence on $N$. Furthermore, let $\beta_h$ be the indicator function for $E_h$, i.e, 
\[
\beta_h(n) = \begin{cases} 1 & \textnormal{ if } n \in E_h \\ 0 & \textnormal{ otherwise.} \\ \end{cases}
\]

All asymptotic notations from this point onward, such as $o, O,$ and $\ll$, are to be interpreted as $N \to \infty$; in particular, the implied constants may depend on $k, r, \eta,$ and $\mathcal H$. We will in multiple instances let $\eps, \epsilon,$ and $\epsilon_1$ denote positive real numbers that one can take to be sufficiently small.

\subsection{Statement of Precise Results}
We now state a precise version of our main theorem.

\begin{theorem}\label{thm:main} Suppose that a subset $\mathcal{P}$ of the primes has positive density $\delta = \delta(\mathcal P)$ in the set of all primes, satisfies a Siegel--Walfisz condition as in \eqref{2.2}, and has a positive level of distribution as in \eqref{2.3}, with $B = B(\mathcal P)$ defined accordingly. Let $\theta_1>0$ be a level of distribution for $\mathcal{P}$, and let $\theta \defeq \min(\half, \theta_1)$. Let $m \ge 2$ be a positive integer, and let $\{ n + h_1, n + h_2, \dots, n + h_k \}$ be a admissible tuple of linear forms. Then for infinitely many positive integers $n$, at least $m$ of $n+h_1, \dots, n+h_k$ are $E_r(\mathcal{P})$ numbers, given that
\begin{equation}\label{accboundforms}
k >  \exp\left(r+\frac{r}{\delta}\left(\frac{2B(m-1)(r-1)!}{\phi(B)\theta(r-1)^{r-1}}\right)^\frac{1}{r}\right).
\end{equation}
\end{theorem}

\begin{remark}
For $\mathcal P = \PP$, we have $\delta = 1$, $B=1$, and $\theta_1 = \half - \eps$ for all small $\eps$.
\end{remark}

Given this theorem, we easily obtain a corollary about bounded gaps between $E_r(\mathcal{P})$ numbers.

\begin{corollary}\label{cor:main} Suppose that a subset $\mathcal{P}$ of the primes has positive density $\delta = \delta(\mathcal P)$ in the set of all primes, satisfies a Siegel--Walfisz condition as in \eqref{2.2}, and has a positive level of distribution as in \eqref{2.3}, with $B = B(\mathcal P)$ defined accordingly. Let $\theta_1>0$ be a level of distribution for $\mathcal{P}$, and let $\theta \defeq \min(\half, \theta_1)$. Then, for $r \geq 2$,
\begin{equation}\label{accbound}
\liminf_{n \to \infty} (a_{n+m} - a_{n}) \ll \mathcal L \log L,
\end{equation}
where $a_{n}$ denotes the $n$-th $E_{r}(\mathcal{P})$ number,
\[
\mathcal L = \exp\left(r+\frac{r}{\delta}\left(\frac{2Bm(r-1)!}{\phi(B)\theta(r-1)^{r-1}}\right)^\frac{1}{r}\right).
\]
\end{corollary}

\begin{proof}
In Theorem \ref{thm:main}, we can let $h_i$ be the $i$-th prime greater than $k$ for $1 \le i \le k.$ It is easy to verify that $\{n + h_1, n + h_2, \dots, n + h_k\}$ is an admissible tuple of linear forms, and that $h_k - h_1 \ll \mathcal{L} \log \mathcal{L}.$
\end{proof}

\begin{remark} A computation using the results of Theorem \ref{thm:main} and Corollary \ref{cor:main} are done in Section \ref{sec:example}.

\end{remark}

\subsection{Outline of Proof}\label{outline}
As in \cite{maynard}, let $D_0 = \log \log \log N$. Let 
\[
W = \prod_{\substack{ \ell < D_0}} \ell.
\]
Note then that $W \ll (\log \log N)^2$. For sufficiently large $N$, $B \mid W$, so we can define $U = \frac{W}{B}$. For each prime $\ell \mid U$, choose a residue class $v_{\ell} \pmod \ell$ such that 
\[
\left(\ell, \prod_{i = 1}^k (v_{\ell}+h_i)\right) = 1.
\]
By the Chinese remainder theorem, we can then choose a residue class $v_0$ of $U$ such that 
\[
\left(U, \prod_{i = 1}^k (v_0+h_i)\right) = 1.
\]

For $\rho > 0$, define
\begin{equation}
S(N, \rho) \defeq \sum_{\substack{N \leq n < 2N \\ n\equiv v_0 \pmod U}} \left(\left( \sum_{i=1}^k \beta_r(n+h_i)\right) - \rho \right)w_n,
\end{equation}
where $w_n$ are nonnegative weights. Suppose we can show $S(N,\rho) > 0$. Then, at least one summand in the sum over $n$ must be positive; for this value of $n$, at least $\floor{\rho}+ 1$ of $n+h_1, \ldots, n+h_k$ are in $E_r$, as desired. It remains to show that for some choice of $w_n$, $S(N, \rho) > 0$ for all sufficiently large $N$.
Let $R \defeq N^{\frac{\theta}{2} - \epsilon}$. We will choose $w_n$ to be the Maynard--Tao weights:
\begin{equation}
w_n \defeq \left(\sum_{d_i \mid n+h_i} \lambda_{d_1, \ldots, d_k} \right)^2,
\end{equation}
where $\lambda_{d_1, \ldots, d_k}$ is supported on $d_1, \ldots, d_k$ for which $\prod_{i=1}^k d_i < R$, $(d_i, W) = 1$ for $1 \leq i \leq k$, and $\mu^2\big(\prod_{i = 1}^k d_i\big) = 1$. We will defer our specific choice of $\lambda_{d_1, \ldots, d_k}$ to be optimized later.

We can write $S(N,\rho) = S_2 - \rho S_1$ for 
\begin{align*}
S_1 &\defeq \sum_{\substack{N \le n < 2N \\ n \equiv v_0 \pmod{U}}} \left (\sum_{d_i \mid n + h_i} \lambda_{d_1, \ldots, d_k} \right)^2, \\
S_2^{(m)} &\defeq \sum_{\substack{N \le n < 2N \\ n \equiv v_0 \pmod{U}}} \left (\sum_{d_i \mid n + h_i} \lambda_{d_1, \ldots, d_k} \right)^2 \beta_r(n + h_m), \\
S_2 &\defeq \sum_{m = 1}^{k} S_2^{(m)}.
\end{align*}
We will now show that $S_2 - \rho S_1 > 0$ for all sufficiently large $N$ for some choice of $\lambda_{d_1,\ldots,d_k}$.

\begin{proposition}\label{prop1}
Let $F$ be a smooth function supported on 
\[
\mathcal{R}_{k, \eta} \defeq \{ (x_1,\ldots, x_k)\in [0, 1]^k \colon x_i \le \frac{2\eta}{\theta}, \sum_{i = 1}^{k} x_i \le 1 \}.
\] 
There exists a corresponding choice of $\lambda_{d_1,\ldots,d_k}$ such that
\begin{equation}\label{s1s1}
S_1 = (1 + o(1)) \frac{\phi(W)^k BN (\log R)^k}{W^{k + 1}} I_k(F),
\end{equation}
and
\begin{equation}\label{s2s2}
S_2^{(m)} = (1 + o(1))\frac{\mathcal{T}_N\phi(B)\phi(W)^k(\log R)^{k + 1}}{W^{k + 1}}J_k^{(m)}(F).
\end{equation}
Here, 
\[
I_k(F) \defeq \int_{0}^{1} \ldots \int_{0}^{1} F(t_1, \ldots, t_k)^2 dt_1 \ldots dt_k,
\]
\[
J_k^{(m)}(F) \defeq \int_{0}^{1} \ldots \int_{0}^{1} \left (\int_{0}^1 F(t_1, \ldots, t_k) dt_m \right)^2 dt_1 \ldots dt_{m-1} dt_{m+1} \ldots dt_k,
\]
and 
\[
\mathcal{T}_N \defeq \sideset{}{^{\dagger}}\sum_{q} X_{N/q},
\]
where $X_n \defeq \sum_{n \le t < 2n} \beta_1(t)$ and $\sum^\dagger$ denotes the restriction that $q = p_1 \cdots p_{r-1} < N^{\half}$ such that $ N^\eta \le p_1 < \cdots < p_{r-1}$. Furthermore, for $0 < \eta < \frac{1}{2(r-1)}$, we have
\begin{equation}
\label{sn}
 \mathcal{T}_N \ge (1+o(1))\frac{\delta^r N}{(\log N)(r-1)!}\left(\log \frac{1}{2(r-1)\eta} \right)^{r-1}.
\end{equation}
\end{proposition}
The proof of Proposition \ref{prop1} depends on a Bombieri--Vinogradov-type result, proved in Section \ref{sec:bv}:
\begin{theorem}
Under the hypotheses of Theorem~\ref{thm:main} on $\mathcal{P}$, the following inequality holds for any integer $r \geq 2$, $u \in \RR$, and $A>0$:
\[
\sum_{\substack{q \le N^{u(\frac{1}{2} - \epsilon)} \\ (q, B) = 1}} \max_{(a, q) = 1} \left | \sum_{n < N^u} \left (\frac{1}{\phi(q)} \beta_r(n; q) - \beta_r(n; q, a) \right ) \right | \ll \frac{N^u}{(\log N)^A},
\]
where $\beta_r(n; q, a)$ denotes $\beta_r$ supported only on $x \equiv a \pmod{q}$, $\beta_r(n; q)$ denotes $\beta_r$ supported only on $(n, q) = 1$, and the implied constant depends only on $u$ and $A$.
\end{theorem}
\eqref{s1s1} was proved in \cite{thorner}. We will prove \eqref{s2s2} in Sections \ref{sec:reduction}, \ref{sec:combinatorial}, and \ref{sec:smoothy}.
\begin{proposition}\label{prop}
Retain the hypotheses on $\mathcal P$ in Theorem~\ref{thm:main}, and let $\mathcal H$ be an admissible set of $k$ distinct linear forms given by $n+h_1, \ldots, n+h_k$. Let $\mathcal{S}_{k,\eta}$ denote the space of smooth functions $F \colon [0,1]^k \to \RR$ supported on $\mathcal{R}_{k,\eta}$, with $I_k(F) \neq 0 $ and $J_k^{(m)}(F)\neq 0$ for $1 \le m \le k$. Define
\[
M_{k,\eta} \defeq \sup_{F \in \mathcal{S}_{k,\eta}} \frac{J_k^{(m)}(F)}{I_k(F)} \hspace{12pt}\text{ and }\hspace{12pt} \nu = \nu_{k,\eta} \defeq \ceil{\frac{\theta \phi(B)\delta^r M_{k,\eta}}{2B(r-1)!}\left(\log \frac{1}{2(r-1)\eta} \right)^{r-1}}
\]
Then, there are infinitely many $n \in \NN$ such that at least $\nu$ of the numbers $n+h_i$ are in $E_r$. In particular,
\[
\liminf_{n \to \infty} (a_{n+\nu-1} - a_{n}) \le h_k - h_1.
\]
\end{proposition}

\begin{proof}
As defined before, $S(N, \rho) = S_2- \rho S_1$.
By the definition of $M_{k,\eta}$, there exists $F \in \mathcal{S}_{k,\eta}$ such that $\sum_{m=1}^k J_k^{(m)}(F) > (M_{k, \eta} - \epsilon_1) I_k(F)$. We assume that $\epsilon_1$ is sufficiently small so that $(M_{k, \eta} - \epsilon_1) I_k(F) > 0$. Therefore, by Proposition \ref{prop1}, there exists a choice of weights $\lambda_{d_1, \dots, d_k}$ such that $S(N, \rho)$ is
{\small
\begin{align*}
 &\ge \frac{\phi(W)N(\log R)^k}{W^{k+1}} \left( \frac{(1+o(1))\phi(B)\delta^r\log R}{B(r-1)!\log N}\left(\log \frac{1}{2(r-1)\eta} \right)^{r-1} \left(\sum_{m=1}^k J_k^{(m)}(F)\right) - \rho I_k(F) + o(1)\right)\\
&\ge \frac{\phi(W)N(\log R)^k I_k(F)}{W^{k+1}} \left( \frac{(1+o(1))\phi(B)\delta^r}{B(r-1)!}\left(\log \frac{1}{2(r-1)\eta} \right)^{r-1}\left(\frac{\theta}{2}-\epsilon \right) \left(M_{k,\eta}-\epsilon_1 \right) - \rho+ o(1)\right).
\end{align*}}
Set 
\[
\rho=\displaystyle\frac{\theta \phi(B)\delta^r M_{k,\eta} }{2B(r-1)!}\left(\log \frac{1}{2(r-1)\eta} \right)^{r-1}-\eps.
\]
By choosing $\epsilon$ and $\epsilon_1$ suitably small in terms of $\eps$, we obtain that $S(N, \rho)>0$ for all sufficiently large $N$. Thus, there are infinitely many $n \in \NN$ such that at least $\floor{\rho}+1$ of the numbers $n+h_i$ are in $E_r$. Since $\floor{\rho}+1 = \nu$ for sufficiently small $\eps$, taking $\eps \to 0$ (and accordingly, $\epsilon, \epsilon_1 \to 0$) gives the result.
\end{proof}

In light of Proposition~\ref{prop}, it remains to find a large lower bound for
\begin{equation}\label{functional}
\sup_{0 < \eta< \frac{1}{2(r-1)} } \sup_{\substack{F \in \mathcal{S}_{k,\eta}}}\frac{\sum_{m=1}^k J_k^{(m)}(F)}{I_k(F)}
\end{equation}
in terms of $k$. In Section~\ref{sec:smoothf}, we make our choice of $\eta$ and $F\in \mathcal{S}_{k,\eta}$ so that 
\[ \frac{\theta \phi(B)\delta^r \sum_{m=1}^k J^{(m)}_k(F)}{2B (r-1)! \cdot I_k(F)} \left(\log \frac{1}{2(r-1)\eta} \right)^{r-1} - \rho > 0 
\]
holds as long as
\begin{equation}\label{doneyo}
k >  \exp\left(r+\frac{r}{\delta}\left(\frac{2B\rho(r-1)!}{\phi(B)\theta(r-1)^{r-1}}\right)^\frac{1}{r}\right).
\end{equation}

\section{Proof of Bombieri--Vinogradov for Products of Primes}\label{sec:bv}

In \cite{thorne}, a Bombieri--Vinogradov result was proved for $E_r(\mathcal{P})$ numbers. The proof uses the following result of Bombieri, Friedlander, and Iwaniec \cite[Equation 1.5]{bfi}, which was inspired by the work of Motohashi \cite{motohashi}.

Recall that the \emph{convolution} of two arithmetic functions $a$ and $b$, denoted $a * b$, is defined as \[ (a * b)(n) = \sum_{d|n} a(d)b(n/d). \]
\begin{lemma} \label{bfi}
Let $x_1, x_2 > 0$, and let $\{a(i)\}$ and $\{b(i)\}$ be two sequences. Let the restriction of these sequences on $[N^{x_1}, 2N^{x_1}]$ and $[N^{x_2}, 2N^{x_2}]$ be $\{a'(i)\}$, $\{b'(i)\}$, respectively. Note that $a', b'$ depend on $N$. Assume that $\{b(i)\}$ satisfies the following condition for any $(d, k) = 1$, $(l, k) = 1$: for any constant $A > 0$,
\[
\Big|\sum_{\substack{n \equiv l \pmod k \\ (n, d) = 1}} b'(n) - \frac{1}{\phi(k)} \sum_{(n, dk) = 1} b'(n)\Big| \ll \| b'\| N^{\frac{1}{2}} \tau(d)^B (\log N)^{-A},
\]
where $\|\cdot\|$ denotes the $L_2$ norm of the sequence. Then, for any constant $A' > 0$, the convolution $a' * b'$ satisfies
\[
\sum_{q \le N^{\frac{x_1 + x_2}{2} - \eps}} \max_{(a, q) = 1} \Big| \frac{1}{\phi(q)}\sum_{\substack{n < N^{x_1 + x_2} \\ (n, q) = 1}} (a' * b')(n) - \sum_{\substack{n < N^{x_1 + x_2} \\ n \equiv a \pmod q}} (a' * b')(n) \Big| \ll \|a'\| \|b'\| \frac{N^{\frac{1}{2}(x_1 + x_2)}}{(\log N)^{A'}}.
\]
\end{lemma}

For our purposes, we will require a Bombieri--Vinogradov theorem for those $E_r(\mathcal{P})$ numbers with prime factors restricted to certain intervals.

\begin{theorem}
\label{bv}
Define
\[
F_r(\mathcal{P}) \defeq \Big\{\prod_{i = 1}^{r} p_i \colon p_i \in \mathcal{P}_i, p_i \textnormal{ all distinct} \Big\},
\]
where $\mathcal{P}_i$ denotes the restriction of $\mathcal{P}$ to some range $[N^{a_i}, N^{b_i}]$. Then, for any $u \in \RR$ and $A > 0$, we have
\[
\sum_{q \le N^{\frac{1}{2}u - \eps}} \max_{(a, q) = 1} \Big| \sum_{n < N^u} \Big(\frac{1}{\phi(q)} \beta_r(n; q) - \beta_r(n; q, a) \Big) \Big| \ll \frac{N^u}{(\log N)^A},
\]
where $\beta_r(n; q, a)$ denotes the indicator function $\beta_r$ for $F_r$ supported on $x \equiv a \pmod{q}$, $\beta_r(n; q)$ denotes $\beta_r$ supported on $(n, q) = 1$, and the implied constant only depending on $u$, $A$, and $\eps$.
\end{theorem}

Suppose first that all pairs of intervals $[a_i, b_i]$ and $[a_j, b_j]$ are either identical or disjoint, and further suppose that $[a_r, b_r]$ and $[a_i, b_i]$ are the same interval for $c$ values of $i \le r$. Define $\Delta_r := \frac{1}{c} \beta_{r-1}*I_{\mathcal{P}_r} - \beta_r$. We can approximate $\beta_r$ with $\beta_{r-1} * I_{\mathcal{P}_r}$ using the Triangle Inequality as follows:
\[
\sum_{q \le N^{\frac{1}{2}u - \eps}} \max_{(a, q) = 1} \Big| \sum_{n < N^u} \Big(\frac{1}{\phi(q)} \beta_r(n; q) - \beta_r(n; q, a) \Big) \Big| 
\]
\[
-\frac{1}{c} \sum_{q \le N^{\frac{1}{2}u - \eps}} \max_{(a, q) = 1} \Big| \sum_{n< N^u} \Big (\frac{1}{\phi(q)} (\beta_{r-1} * I_{\mathcal{P}_r}(n; q) - \beta_{r-1} * I_{\mathcal{P}_r}(n; q, a) \Big) \Big|
 \]
\[
\le \sum_{q \le N^{\frac{1}{2}u - \eps}} \max_{(a, q) = 1} \Big| \sum_{n < N^u} \Big (\frac{1}{\phi(q)} \Delta_r(n; q) - \Delta_r(n; q, a) \Big) \Big|
\]
Note that $\Delta_r$ only takes nonzero values at those integers divisible by the square of a prime in the interval $[N^{a_r}, N^{b_r}]$, and it has a maximum magnitude of at most 1 at each integer. Fix $q$. For every number $p$ in $[N^{a_r}, N^{b_r}]$, at most $\frac{N^u}{p^2}$ numbers less than $N^u$ are divisible by $p^2$. Thus, for each $q$, we have
\[
\Big| \sum_{n < N^u} \frac{1}{\phi(q)}\Delta_r(n; q) \Big| \le \frac{N^u}{\phi(q)} \sum_{\substack{p \in [N^{a_r}, N^{b_r}] \\ p \textnormal{ prime}}} \frac{1}{p^2} \ll \frac{1}{\phi(q)}N^{u - a_r}
\]
Summing over all $q$, we obtain
\begin{equation}\label{toerr}
\sum_{q \le N^{\frac{1}{2}u - \eps}} \Big| \sum_{n < N^u} \frac{1}{\phi(q)}\Delta_r(n ; q) \Big | \ll N^{u - a_r} \sum_{q \le N^{\frac{1}{2}u - \eps}} \frac{1}{\phi(q)} \ll N^{u - a_r}(\log N)(\log \log N),
\end{equation}
where we have used the bound $\frac{1}{\phi(q)} \ll \frac{\log \log q}{q}$.

Moreover, for every number $p$ in $[N^{a_r}, N^{b_r}]$, at most $\frac{N^u}{p^2q} + 1$ numbers less than $N^u$ are divisible by $p^2$ and are congruent to $a \pmod{q}$. Thus, a similar approach gives
\begin{equation}\label{isto}
\sum_{q \le N^{\frac{1}{2}u - \eps}} \Big| \sum_{n < N^u} \Delta_r(n; q, a) \Big| \ll N^{u - a_r}(\log N).
\end{equation}
The two error terms given by \eqref{toerr} and \eqref{isto} are negligible, so it remains to show the desired equidistribution property for the sequence $\beta_{r-1} * I_{\mathcal{P}_r}$. We appeal to Lemma~\ref{bfi}. A direct application would result in a bound of $N^{1.5u}$ instead of $\frac{N}{(\log N)^A}$, which is not strong enough for our purposes. We instead decompose our sum into dyadic intervals:
\[
\beta_{r-1} * I_{\mathcal{P}_r}(n) = \sum_{t \in [1,n]} \beta_{r-1}(t) I_{\mathcal{P}_r}(n/t)
\]
\[
= \sum_{\omega}\sum_{t \in J_{\omega}} \beta_{r-1}(t) I_{\mathcal{P}_r}(n/t)
\]
for the cover $J_1, J_2, \ldots$ of $[1, n]$, $J_{\omega} \defeq [2^{\omega}, 2^{\omega + 1})$. For each $J_{\omega}$, define the interval $K_{\omega} = [1, \frac{N^u}{2^{\omega}}]$. Let the restriction of $\beta_{r-1}(x)$ and $I_{\mathcal{P}_r}$ to the intervals $J_\omega$ and $K_\omega$ be $\beta_{r-1, \omega}(x)$ and $I_{\mathcal{P}_r, \omega}$, respectively. Note that \[ \sum_{x < N^u} \beta_{r-1, \omega} * I_{\mathcal{P}_r} \text{ and } \sum_{x < N^u} \beta_{r-1, \omega} * I_{\mathcal{P}_r, \omega} \] are identical. By applying Lemma-\ref{bfi} to the sequences $\beta_{r-1, \omega}$ and $I_{\mathcal{P}_r, \omega}$ and splitting the support of the second sequence into $\frac{\log(N^u)}{\log 2} + O(1)$ dyadic intervals, we obtain the bound
\[
\sum_{q \le N^{\frac{1}{2}u - \eps}} \max_{(a, q) = 1} \Big| \sum_{n < N^u} \Big(\frac{1}{\phi(q)} (\beta_{r-1, \omega} * I_{\mathcal{P}_r, \omega})(n; q) - (\beta_{r-1, \omega} * I_{\mathcal{P}_r, \omega})(n; q, a) \Big) \Big|
\]
\[
\ll \sqrt{2^\omega} \sqrt{\frac{N^u}{2^{\omega}}} \frac{N^{\frac{1}{2}u}}{(\log N)^A} \ll \frac{N^u}{(\log N)^A}
\]
The total number of intervals needed to cover all of $[1, N^u]$ is $\frac{\log N^u}{\log 2} + O(1)$. Therefore, the sum of the above expression over all $\omega$ is $\ll \frac{N^u}{(\log N)^{A'}}$ for any $A'>0$, as desired.

\section{Reduction to Counting Prime Numbers}\label{sec:reduction}
We reduce $S_2^{(m)}$ to the following sum involving the prime counting function $X_n$.
\begin{lemma} We have that
\begin{equation}\label{ninja}
S_2^{(m)}=\frac{\phi(B)}{\phi(W)}\sideset{}{^\dagger}\sum_{q} X_{N/q} \sideset{}{'}\sum_{\substack{d_1, \ldots, d_k \\ e_1, \ldots, e_k \\ d_m, e_m \mid q}}\frac{\lambda_{d_1, \ldots, d_k}\lambda_{e_1, \ldots, e_k}}{\prod_{i \neq m} \phi([d_i, e_i])} + O\Big(\lambda_{\max}^2\frac{N}{(\log N)^A}\Big),
\end{equation}
where $\sum^\dagger$ is defined in Proposition \ref{prop1} and $\lambda_{\max} \defeq \max_{d_1,\ldots,d_k} \abs{\lambda_{d_1, \ldots, d_k}}$.
\end{lemma}
\begin{proof}
Fix $m \in \{ 1, \ldots, k\}.$ We first switch the order of summation for $S_2^{(m)}$ to obtain
\begin{align*}
S_2^{(m)} = &\sum_{\substack{N \le n < 2N \\ n \equiv v_0 \pmod{U}}} \Big (\sum_{d_i \mid n + h_i} \lambda_{d_1, \ldots, d_k} \Big)^2 \beta_r(n + h_m) \\&= \sum_{\substack{d_1, \ldots, d_k \\ e_1, \ldots, e_k}}\lambda_{d_1, \ldots, d_k}\lambda_{e_1, \ldots, e_k} \sum_{\substack{N \le n < 2N \\ n \equiv v_0 \pmod{U} \\ [d_i, e_i] \mid n+h_i}} \beta_r(n+h_m).
\end{align*}
For sufficiently large $N$, the inner sum vanishes if $U, [d_1, e_1] , \ldots, [d_k,e_k]$ are not pairwise coprime. Indeed, if $p \mid U$ and $p \mid [d_i, e_i]$, then either $\lambda_{d_1, \dots, d_k} = 0$ or $\lambda_{e_1, \dots, e_k} = 0$, as $\lambda_{d_1, \dots, d_k}$ is supported only when $\mu^2(BU \prod_{i = 1}^m d_i) = 1.$ Now assume $p \nmid W$. If $p\mid [d_i, e_i]$ and $p \mid [d_j, e_j]$, then $p \mid n+h_i$ and $p \mid n+h_j.$ This implies that $p \mid h_i-h_j$, which contradicts  $D_0 > h_k-h_1$ (as $D_0$ grows with $N$). Thus, by the Chinese remainder theorem,
\[
S_2^{(m)} = \sideset{}{'}\sum_{\substack{d_1, \ldots, d_k \\ e_1, \ldots, e_k}}\lambda_{d_1, \ldots, d_k}\lambda_{e_1, \ldots, e_k} \sum_{\substack{N \le n < 2N \\ n \equiv v_0 \pmod{U} \\ [d_i, e_i] \mid n+h_i}} \beta_r(n+h_m),
\]
where $\sum^{'}$ denotes the restriction that $U, [d_1, e_1] , \ldots, [d_k,e_k]$ are pairwise coprime.

Note that $[d_m, e_m]$ divides $n+h_m$ while $d_m, e_m \le R \le N^{\half}$. So, when $\beta_r(n+h_m)=1$, $[d_m,e_m]$ is composed of at most $r-1$ prime factors that are at least $N^\eta$ but at most $N^\half$.
Thus, we can write $S_2^{(m)} = \sum_{h = 1}^{r} T_h$, where 
\begin{align*}
T_h \defeq &\sideset{}{'}\sum_{\substack{d_1, \ldots, d_k \\ e_1, \ldots, e_k \\ [d_m, e_m] \in E'_{r-h}}}\lambda_{d_1, \ldots, d_k}\lambda_{e_1, \ldots, e_k} \sum_{\substack{N \le n < 2N \\ n \equiv v_0 \pmod{U} \\ [d_i, e_i] \mid n+h_i}} \beta_r(n+h_m)
 \\&=\sideset{}{'}\sum_{\substack{d_1, \ldots, d_k \\ e_1, \ldots, e_k \\ [d_m, e_m] \in E'_{r-h}}}\lambda_{d_1, \ldots, d_k}\lambda_{e_1, \ldots, e_k} \sum_{\substack{N \le n < 2N \\ n \equiv v \pmod{U\prod_{i \neq m}[d_i,e_i]} \\ [d_m,e_m] \mid n+h_m}} \beta_r(n+h_m),
\end{align*}
where $E'_{r-h} \defeq \{p_1 \cdots p_{r-h} \colon N^\eta \le p_1 < \cdots < p_{r-h} \}$, and in the last step, we have used the Chinese remainder theorem to convert the congruence conditions mod $U$ and 
mod $[d_i,e_i]$ for $i\neq m$ into a single congruence $v \pmod {U \prod_{i \neq m}[d_i,e_i]}$. For $n$ such that $[d_m, e_m] \mid n+h_m$, we can make the change of variables $t = \frac{n+h_m}{[d_m,e_m]}$ to obtain
\[
T_h = M_h + \Xi_h,
\]
where 
\[
M_h \defeq \sideset{}{'}\sum_{\substack{d_1, \ldots, d_k \\ e_1, \ldots, e_k \\ [d_m, e_m] \in E'_{r-h}}}\lambda_{d_1, \ldots, d_k}\lambda_{e_1, \ldots, e_k} \sum_{\substack{\frac{N+h_m}{[d_m, e_m]} \le t < \frac{2N+h_m}{[d_m, e_m]} \\ t \equiv v' \pmod{U\prod_{i \neq m}[d_i,e_i]}}} \beta_h(t),
\]
$v' \pmod{U\prod_{i \neq m}[d_i,e_i]}$ is the residue class resulting from the change of variables, and $\Xi_h$ denotes the error resulting from $t$ such that $\beta_h(t) = 1$ but $\beta_r(n+h_m)=0$ for the corresponding $n$. The latter occurs precisely when at least one prime dividing $[d_m, e_m]$ also divides $t$, since then the corresponding $n + h_m$ is not the product of distinct primes. We bound the error by
\begin{align*}
\abs{\Xi_h} &\le \sideset{}{'}\sum_{\substack{d_1, \ldots, d_k \\ e_1, \ldots, e_k \\ [d_m, e_m] \in E'_{r-h}}}\abs{\lambda_{d_1, \ldots, d_k}\lambda_{e_1, \ldots, e_k}} \sum_{p \mid [d_m,e_m]}\sum_{\substack{\frac{N+h_m}{[d_m, e_m]} \le t < \frac{2N+h_m}{[d_m, e_m]} \\ t \equiv v' \pmod{U\prod_{i \neq m}[d_i,e_i]} \\ p \mid t}} 1 
\\&\ll \lambda_{\max}^2(r-h) \sideset{}{'}\sum_{\substack{d_1, \ldots, d_k \\ e_1, \ldots, e_k \\ [d_m, e_m] \in E'_{r-h} \\ ([d_i, e_i] ,B) = 1}} \frac{N}{N^\eta \cdot U\prod_{i=1}^m [d_i, e_i]} \ll \lambda_{\max}^2N^{1-\eta} \sum_{\substack{\omega \le R^2U\\ (\omega, B) = 1}} \frac{\tau_{3k+1}(\omega)}{\omega} \ll \lambda_{\max}^2N^{1-\eta+\eps},
\end{align*}
where we have used the fact that $U\prod_{i \neq m}[d_i,e_i] \le R^2U$, since $\prod_{i=1}^m d_i, \prod_{i=1}^m e_i < R$. Also, the fact that $p \mid [d_m,e_m]$ implies that $p \nmid U\prod_{i \neq m} [d_i,e_i]$, which allows us to conclude that
\[
\sum_{\substack{\frac{N+h_m}{[d_m, e_m]} \le t < \frac{2N+h_m}{[d_m, e_m]} \\ t \equiv v' \pmod{U\prod_{i \neq m}[d_i,e_i]} \\ p \mid t}} 1 = \frac{1}{p U\prod_{i \neq m}[d_i,e_i]} \cdot \frac{N}{[d_m,e_m]} + O(1).
\]

Next, we remove the congruence condition from the inner sum of $M_h$. To this end, we rearrange the sum as
\begin{align}
M_h &= \sideset{}{'}\sum_{\substack{d_1, \ldots, d_k \\ e_1, \ldots, e_k \\ [d_m, e_m] \in E'_{r-h}}}\lambda_{d_1, \ldots, d_k}\lambda_{e_1, \ldots, e_k} \sum_{\substack{\frac{N+h_m}{[d_m, e_m]} \le t < \frac{2N+h_m}{[d_m, e_m]} \\ t \equiv v' \pmod{U\prod_{i \neq m}[d_i,e_i]}}} \beta_h(t)\nonumber \\ 
&=\sideset{}{'}\sum_{\substack{d_1, \ldots, d_k \\ e_1, \ldots, e_k \\ [d_m, e_m] \in E'_{r-h}}} \frac{\lambda_{d_1, \ldots, d_k}\lambda_{e_1, \ldots, e_k}}{\phi(U)\prod_{i \neq m} \phi([d_i, e_i])} \sum_{\substack{\frac{N+h_m}{[d_m, e_m]} \le t < \frac{2N+h_m}{[d_m, e_m]} \\ (t,U\prod_{i \neq m} [d_i,e_i])=1 } }\beta_h(t)\label{realmain}\\
& + O\Big(\sideset{}{'}\sum_{\substack{d_1, \ldots, d_k \\ e_1, \ldots, e_k \\ [d_m, e_m] \in E'_{r-h}}} \lambda_{\max}^2 \mathcal{E}_{h}\Big( \frac{N}{[d_m,e_m]} , U\prod_{i \neq m}[d_i,e_i]\Big)\Big) \label{realerror},
\end{align}
where 
\[
\mathcal{E}_{h}(V,q) \defeq 1 + \max_{(a,q)=1} \Big| \sum_{\substack{V \le n < 2V \\ n \equiv a \pmod q} }\beta_{h}(n) - \frac{1}{\phi(q)} \sum_{\substack{V \le n < 2V \\ (n,q)=1} }\beta_{h}(n) \Big|.
\]
We can write the expression inside the error term as
\begin{equation}
\lambda_{\max}^2 \sideset{}{'}\sum_{\substack{d_m,e_m \\ [d_m, e_m] \in E'_{r-h}}} \sideset{}{'}\sum_{\substack{d_1, \ldots, d_{m-1}, d_{m+1}, \ldots d_k \\ e_1, \ldots, e_{m-1}, e_{m+1}, \ldots e_k }} \mathcal{E}_{h}\Big(\frac{N}{[d_m, e_m]}, U\prod_{i \neq m} [d_i, e_i] \Big).\label{outererror}
\end{equation}
We further bound the inner sum above by
\begin{align}
& \sideset{}{'}\sum_{\substack{d_1, \ldots, d_{m-1}, d_{m+1}, \ldots d_k \\ e_1, \ldots, e_{m-1}, e_{m+1}, \ldots e_k}} \mathcal{E}_{h}\Big(\frac{N}{[d_m, e_m]}, U\prod_{i \neq m} [d_i, e_i] \Big) \ll \sum_{\substack{\omega \le \frac{R^2U}{[d_m, e_m]} \\ (\omega,B)=1} } \tau_{3k-2}(\omega)\mathcal{E}_{h}\Big(\frac{N}{[d_m, e_m]}, \omega \Big) \nonumber \\
&\ll \left(\sum_{\substack{\omega \le \frac{R^2U}{[d_m, e_m]} \\ (\omega,B)=1} } \tau_{3k-2}(\omega)^2 \frac{N}{[d_m, e_m]\phi(\omega)} \right)^\frac{1}{2} \left( \sum_{\substack{\omega \le \frac{R^2U}{[d_m, e_m]} \\ (\omega,B)=1} } \mathcal{E}_{h}\Big(\frac{N}{[d_m, e_m]}, \omega \Big) \right)^\frac{1}{2}, \label{bvapply}
\end{align}
where we have applied $U\prod_{i \neq m} [d_i,e_i] \le \frac{R^2U}{[d_m,e_m]}$, the Cauchy--Schwarz inequality, and the trivial bound
\[
\mathcal{E}_{h}\Big(\frac{N}{[d_m, e_m]}, \omega \Big) \ll \frac{N}{[d_m,e_m]\phi(\omega)}.
\] 
Since $R^2U/[d_m, e_m] \ll N$, we have that
\[
\sum_{\substack{\omega \le \frac{R^2U}{[d_m, e_m]} \\ (\omega,B)=1} } \frac{\tau_{3k-2}(\omega)^2}{\phi(\omega)} \ll \Big(\log \frac{R^2U}{[d_m, e_m]}\Big)^{a(k)} \ll (\log N)^{a(k)}
\]
for some $a(k) \in \NN$. This shows that
\[
\Big(\sum_{\substack{\omega \le \frac{R^2U}{[d_m, e_m]} \\ (\omega,B)=1} } \tau_{3k-2}(\omega)^2 \frac{N}{[d_m, e_m]\phi(\omega)} \Big)^\frac{1}{2} \ll (\log N)^{\frac{a(k)}{2}}\Big(\frac{N}{[d_m,e_m]}\Big)^{\half}. 
\]
To bound the second factor in \eqref{bvapply}, we use the Bombieri--Vinogradov theorem for $h=1$ and Theorem~\ref{bv} for $h>1$. To justify this step, we check that
\[
\frac{R^2}{[d_m,e_m]} \le \Big(\frac{N}{[d_m,e_m]}\Big)^{\theta - 2\epsilon},
\]
for the case $h=1$. As for the case $h>1$, we note that the congruence conditions $n+h_m \equiv v_0+h_m \pmod U$ and $n+h_m \equiv h_m - h_i \pmod{[d_i,e_i]}$ for $i \neq m$ are primitive for sufficiently large $N$ and thus $W$. Thus, observing that $\frac{N}{[d_m, e_m]}$ and $\frac{2N}{[d_m, e_m]}$ are expressible as $N^u$ for some $u\in \RR$, we can  apply Theorem~\ref{bv}, as desired. Thus, we have for $1 \le h \le r-1$ that
\[
\sum_{\substack{q \le \frac{R^2U}{[d_m, e_m]} \\ (q,B)=1} } \mathcal{E}_{h}\Big(\frac{N}{[d_m, e_m]}, q \Big) \ll \frac{N}{[d_m, e_m](\log N)^{A}}
\]
for any $A > 0$. This allows us to bound \eqref{outererror} by
\[
\lambda_{\max}^2\sum_{\substack{d_m, e_m\\ [d_m,e_m]\in E'_{r-h}}} \frac{N}{[d_m,e_m](\log N) ^{A'} } \le \lambda_{\max}^2\sum_{\substack{q \le R^2\\ q \in E'_{r-h}}} 3^{r-h} \frac{1}{q} \frac{N}{(\log N)^{A'}} \ll \lambda_{\max}^2\frac{N}{(\log N)^{A}},
\]
where $A>0$ can be arbitrary and $A' > A$. Thus, the error \eqref{realerror} is $O\left(\lambda_{\max}^2\frac{N}{(\log N)^{A}}\right)$ for any $A>0$. 

Furthermore, the difference between \eqref{realmain} and
\begin{equation}\label{ughmain}
\sideset{}{'}\sum_{\substack{d_1, \ldots, d_k \\ e_1, \ldots, e_k \\ [d_m, e_m] \in E'_{r-h}}} \frac{\lambda_{d_1, \ldots, d_k}\lambda_{e_1, \ldots, e_k}}{\phi(U)\prod_{i \neq m} \phi([d_i, e_i])} \sum_{\frac{N+h_m}{[d_m, e_m]} \le t < \frac{2N+h_m}{[d_m, e_m]}}\beta_h(t)
\end{equation}
is small. Indeed, 
\begin{align}
&\Biggl\lvert\sideset{}{'}\sum_{\substack{d_1, \ldots, d_k \\ e_1, \ldots, e_k \\ [d_m, e_m] \in E'_{r-h}}} \frac{\lambda_{d_1, \ldots, d_k}\lambda_{e_1, \ldots, e_k}}{\phi(U)\prod_{i \neq m} \phi([d_i, e_i])} \sum_{\frac{N+h_m}{[d_m, e_m]} \le t < \frac{2N+h_m}{[d_m, e_m]}}\beta_h(t) \label{range}
\\ 
&\hspace{24pt }-\sideset{}{'}\sum_{\substack{d_1, \ldots, d_k \\ e_1, \ldots, e_k \\ [d_m, e_m] \in E'_{r-h}}} \frac{\lambda_{d_1, \ldots, d_k}\lambda_{e_1, \ldots, e_k}}{\phi(U)\prod_{i \neq m} \phi([d_i, e_i])} \sum_{\substack{\frac{N+h_m}{[d_m, e_m]} \le t < \frac{2N+h_m}{[d_m, e_m]} \\ (t,U\prod_{i \neq m} [d_i,e_i])=1 } }\beta_h(t)\Biggr\rvert\nonumber
\\ 
&\le \sideset{}{'}\sum_{\substack{d_1, \ldots, d_k \\ e_1, \ldots, e_k \\ [d_m, e_m] \in E'_{r-h}}} \frac{\abs{\lambda_{d_1, \ldots, d_k}\lambda_{e_1, \ldots, e_k}}}{\phi(U)\prod_{i \neq m} \phi([d_i, e_i])} \sum_{p \mid U \prod_{i \neq m} [d_i, e_i]}\sum_{\substack{\frac{N+h_m}{[d_m, e_m]} \le t < \frac{2N+h_m}{[d_m, e_m]} \\ p \mid t}} \beta_h(t). \nonumber
\end{align}
Note that if $p \mid U$, then $p < N^{\eta}$, and since $\beta_h$ is supported on integers with prime factors $\geq N^\eta$, the innermost sum above is empty in this case. Thus, we only need to consider $p \mid \prod_{i \neq m} [d_i,e_i]$. Since $\prod_{i \neq m} [d_i,e_i] \le R^2 < N$, there can only be at most $\floor{1/\eta}$ prime factors for every $\prod_{i \neq m} [d_i,e_i]$, which shows that \eqref{range} is
\begin{align*}
&\le \sideset{}{'}\sum_{\substack{d_1, \ldots, d_k \\ e_1, \ldots, e_k \\ [d_m, e_m] \in E'_{r-h}}} \frac{\abs{\lambda_{d_1, \ldots, d_k}\lambda_{e_1, \ldots, e_k}}}{\phi(U)\prod_{i \neq m} \phi([d_i, e_i])} \sum_{p \mid \prod_{i \neq m} [d_i, e_i]}\sum_{\substack{\frac{N+h_m}{[d_m, e_m]} \le t < \frac{2N+h_m}{[d_m, e_m]} \\ p \mid t}} 1 \\
&\le \lambda_{\max}^2 \floor{\frac{1}{\eta}}\sideset{}{'}\sum_{\substack{d_1, \ldots, d_k \\ e_1, \ldots, e_k \\ [d_m, e_m] \in E'_{r-h}}} \frac{1}{\phi(U)\prod_{i \neq m} \phi([d_i, e_i])} \left( \frac{N}{N^\eta [d_m,e_m]} + O(1) \right)\\
&\le \lambda_{\max}^2 \floor{\frac{1}{\eta}}\sideset{}{'}\sum_{\substack{d_1, \ldots, d_k \\ e_1, \ldots, e_k \\ [d_m, e_m] \in E'_{r-h}}} \frac{1}{\phi(U)\prod_{i \neq m} \phi([d_i, e_i])} \cdot \frac{N}{N^\eta [d_m,e_m]} \\
&\hspace{36pt}+ O\Big(\lambda_{\max}^2 \sideset{}{'}\sum_{\substack{d_1, \ldots, d_k \\ e_1, \ldots, e_k \\ [d_m, e_m] \in E'_{r-h}}} \frac{1}{\phi(U)\prod_{i \neq m} \phi([d_i, e_i])}  \Big)\\
& \le \lambda_{\max}^2\floor{\frac{1}{\eta}} N^{1-\eta} \sum_{\omega \le R^2W} \frac{\tau_{3k}(\omega)}{\phi(\omega)} + O\Big(\lambda_{\max}^2\sideset{}{'}\sum_{\substack{d_1, \ldots, d_k \\ e_1, \ldots, e_k \\ [d_m, e_m] \in E'_{r-h}}} \frac{R^2}{\phi(U)\prod_{i =1}^m \phi([d_i, e_i])}\Big)\\
&\ll \lambda_{\max}^2N^{1-\eta+\eps} + \lambda_{\max}^2N^{\theta-2\epsilon+\eps},
\end{align*}
which is negligible.

Now that we have removed the coprimality condition from the inner sum of \eqref{realmain}, we can rearrange and express our new sum \eqref{ughmain} in terms of $\beta_1$. For a fixed choice of $d_1, \ldots, d_k, e_1, \ldots, e_k$, consider $t$ in the range $\frac{N+h_m}{[d_m,e_m]} \le t < \frac{2N+h_m}{[d_m,e_m]}$ such that $\beta_r(t) = 1$. If the prime factors of $t$ that are less than or equal to $N^{\half}$ are all distinct from the prime factors of $[d_m,e_m]$, we can make a change of variables that divides $t$ by these prime factors to obtain its largest prime factor $u$, obtaining the sum
\begin{equation}
\sideset{}{'}\sum_{\substack{d_1, \ldots, d_k \\ e_1, \ldots, e_k}} \frac{\lambda_{d_1, \ldots, d_k}\lambda_{e_1, \ldots, e_k}}{\phi(U)\prod_{i \neq m} \phi([d_i, e_i])} \sum_{\substack{q = p_1 \cdots p_{r-1} < N^{\half} \\ N^\eta \le p_1 < \cdots < p_{r-1} \\ d_m, e_m \mid q\\ [d_m, e_m] \in E'_{r-h}} } \sum_{\frac{N+h_m}{q} \le u < \frac{2N+h_m}{q}} \beta_1(u). \label{reduced}
\end{equation}
We can bound the error resulting from repeated prime factors dividing both $[d_m,e_m]$ and $t$ as follows:
\begin{align*}
\Biggl\lvert &\sideset{}{'}\sum_{\substack{d_1, \ldots, d_k \\ e_1, \ldots, e_k \\ [d_m, e_m] \in E'_{r-h}}} \frac{\lambda_{d_1, \ldots, d_k}\lambda_{e_1, \ldots, e_k}}{\phi(U)\prod_{i \neq m} \phi([d_i, e_i])} \sum_{\frac{N+h_m}{[d_m, e_m]} \le t < \frac{2N+h_m}{[d_m, e_m]}} \beta_r(t)\\
&- \sideset{}{'}\sum_{\substack{d_1, \ldots, d_k \\ e_1, \ldots, e_k}} \frac{\lambda_{d_1, \ldots, d_k}\lambda_{e_1, \ldots, e_k}}{\phi(U)\prod_{i \neq m} \phi([d_i, e_i])} \sum_{\substack{q = p_1 \cdots p_{r-1} < N^{\half} \\ N^\eta \le p_1 < \cdots < p_{r-1} \\ d_m, e_m \mid q\\ [d_m, e_m] \in E'_{r-h}} } \sum_{\frac{N+h_m}{q} \le u < \frac{2N+h_m}{q}} \beta_1(u) \Biggr\rvert \\
&\le \sideset{}{'}\sum_{\substack{d_1, \ldots, d_k \\ e_1, \ldots, e_k \\ [d_m, e_m] \in E'_{r-h}}} \frac{\abs{\lambda_{d_1, \ldots, d_k}\lambda_{e_1, \ldots, e_k}}}{\phi(U)\prod_{i \neq m} \phi([d_i, e_i])} \sum_{p \mid [d_m,e_m]} \sum_{\frac{N+h_m}{p[d_m, e_m]} \le t < \frac{2N+h_m}{p[d_m, e_m]}}1 \\
&\le \sideset{}{'}\sum_{\substack{d_1, \ldots, d_k \\ e_1, \ldots, e_k \\ [d_m, e_m] \in E'_{r-h}}} \frac{\abs{\lambda_{d_1, \ldots, d_k}\lambda_{e_1, \ldots, e_k}}}{\phi(U)\prod_{i \neq m} \phi([d_i, e_i])}(r-h) \Big( \frac{N}{N^\eta[d_m,e_m]}+O(1)\Big)
\end{align*}
Just as before, the error term resulting from $\frac{N}{N^\eta[d_m,e_m]}$ is $\ll \lambda_{\max}^2N^{1-\eta+\eps}$, and the error term resulting from the $O(1)$ is $\ll \lambda_{\max}^2 R^2N^{\eps}$.

Summing \eqref{reduced} and the dominating error term of $O(\lambda_{\max}^2 \frac{N}{(\log N)^A})$ across all $h$, we obtain
{\small \begin{equation}
S_2^{(m)} = \sideset{}{'}\sum_{\substack{d_1, \ldots, d_k \\ e_1, \ldots, e_k}} \frac{\lambda_{d_1, \ldots, d_k}\lambda_{e_1, \ldots, e_k}}{\phi(U)\prod_{i \neq m} \phi([d_i, e_i])} \sum_{\substack{q = p_1 \cdots p_{r-1} < N^{\half} \\ N^\eta \le p_1 < \cdots < p_{r-1} \\ d_m, e_m \mid q} } \sum_{\frac{N+h_m}{q} \le v < \frac{2N+h_m}{q}} \beta_1(v)  + O\left(\frac{N\lambda_{\max}^2}{(\log N)^A}\right). \label{finale}
\end{equation}}
Since 
\[\sum_{\frac{N+h_m}{q} \le v < \frac{2N+h_m}{q}} \beta_1(v) = \sum_{\frac{N}{q} \le v < \frac{2N}{q}} \beta_1(v) + O\left(\frac{1}{q}\right)
\]
and
\begin{align*}
\sideset{}{'}\sum_{\substack{d_1, \ldots, d_k \\ e_1, \ldots, e_k}}&\frac{\lambda_{d_1, \ldots, d_k}\lambda_{e_1, \ldots, e_k}}{\phi(U)\prod_{i \neq m} \phi([d_i, e_i])} \sum_{\substack{q = p_1 \cdots p_{r-1} < N^{\half} \\ N^\eta \le p_1 < \cdots < p_{r-1} \\ d_m, e_m \mid q} } O\left(\frac{1}{q}\right) \\
&\le \sideset{}{'}\sum_{\substack{d_1, \ldots, d_k \\ e_1, \ldots, e_k}}\frac{\lambda_{d_1, \ldots, d_k}\lambda_{e_1, \ldots, e_k}}{\phi(U)\prod_{i \neq m} \phi([d_i, e_i])} O(1) \ll \lambda_{\max}^2 R^2 N^{\eps} \ll \lambda_{\max}^2\frac{N}{(\log N)^A},
\end{align*}
we can replace the sum in \eqref{finale} with
\[
\frac{1}{\phi(U)}\sideset{}{^\dagger}\sum_{q} X_{N/q} \sideset{}{'}\sum_{\substack{d_1, \ldots, d_k \\ e_1, \ldots, e_k \\ d_m, e_m \mid q}}\frac{\lambda_{d_1, \ldots, d_k}\lambda_{e_1, \ldots, e_k}}{\prod_{i \neq m} \phi([d_i, e_i])} = \frac{\phi(B)}{\phi(W)}\sideset{}{^\dagger}\sum_{q} X_{N/q} \sideset{}{'}\sum_{\substack{d_1, \ldots, d_k \\ e_1, \ldots, e_k \\ d_m, e_m \mid q}}\frac{\lambda_{d_1, \ldots, d_k}\lambda_{e_1, \ldots, e_k}}{\prod_{i \neq m} \phi([d_i, e_i])}.
\] 
\end{proof}

\section{Combinatorial Sieve Manipulations}\label{sec:combinatorial}

In this section, we will only consider the innermost sum of \eqref{ninja}, i.e. \[ \sideset{}{'}\sum_{\substack{d_1, \ldots, d_k \\ e_1, \ldots, e_k \\ d_m, e_m \mid q}}\frac{\lambda_{d_1, \ldots, d_k}\lambda_{e_1, \ldots, e_k}}{\prod_{i \neq m} \phi([d_i, e_i])}. \] The manipulation done here will be similar to that of \cite[Lemmas 5.2, 5.3]{maynard}.

\begin{lemma}
Define 
\begin{equation} \label{yeezy} y_{r_1, \ldots, r_m}^{(m, q)} \defeq \prod_{i \neq m} \mu(r_i)g(r_i) \sum_{\substack{d_1, \ldots, d_k \\ r_i \mid d_i \\ d_m \mid q}} \frac{\lambda_{d_1, \ldots, d_k}}{\prod_{i \neq m} \phi(d_i)}, \end{equation}
where $g$ denotes the totally multiplicative function defined on the primes by $g(\ell) = \ell-2$. Then, we have
\[ \sideset{}{'}\sum_{\substack{d_1, \ldots, d_k \\ e_1, \ldots, e_k \\ d_m, e_m \mid q}}\frac{\lambda_{d_1, \ldots, d_k}\lambda_{e_1, \ldots, e_k}}{\prod_{i \neq m} \phi([d_i, e_i])} = \sum_{\substack{r_1, \ldots, r_k \\ r_m = 1}} \frac{(y_{r_1, \ldots, r_k}^{(m, q)})^2}{\prod_{i \neq m}g(r_i)} + O\left(\frac{(y^{(m, q)}_{\max})^2\phi(W)^{k-1} (\log R)^{k-1}}{W^{k-1}D_0} \right), \] where \[ y^{(m, q)}_{\max} = \max_{r_1, \ldots, r_k} |y_{r_1, \ldots, r_k}^{(m, q)}|. \]
\end{lemma}
Note that $y_{r_1, \ldots, r_k}^{(m, q)}$ is supported on $(r_1, \ldots, r_k)$ such that $r_m \mid q.$ 

\begin{proof}
For squarefree $d_i$ and $e_i$ we have the identity \[ \frac{1}{\phi([d_i, e_i])} = \frac{1}{\phi(d_i)\phi(e_i)}\sum_{u_i \mid d_i, e_i} g(u_i). \] Therefore, \[ \sideset{}{'}\sum_{\substack{d_1, \ldots, d_k \\ e_1, \ldots, e_k \\ d_m, e_m \mid q}}\frac{\lambda_{d_1, \ldots, d_k}\lambda_{e_1, \ldots, e_k}}{\prod_{i \neq m} \phi([d_i, e_i])} = \sideset{}{'}\sum_{\substack{d_1, \ldots, d_k \\ e_1, \ldots, e_k \\ d_m, e_m \mid q}}\frac{\lambda_{d_1, \ldots, d_k}\lambda_{e_1, \ldots, e_k}}{\prod_{i \neq m} \phi(d_i)\phi(e_i)} \sum_{\substack{u_1, \ldots, u_k \\ u_i \mid d_i, e_i \\ u_m = 1}} \prod_{i \neq m} g(u_i). \] Swapping the order of summation between the $u$ and $d$ variables gives that the previous quantity is equal to \[ \sum_{\substack{u_1, \ldots, u_k \\ u_m = 1}} \prod_{i \neq m} g(u_i) \sideset{}{'}\sum_{\substack{d_1, \ldots, d_k \\ e_1, \ldots, e_k \\ u_i \mid d_i, e_i \\ d_m, e_m \mid q}}\frac{\lambda_{d_1, \ldots, d_k}\lambda_{e_1, \ldots, e_k}}{\prod_{i \neq m} \phi(d_i)\phi(e_i)}. \]

Let us now consider the conditions that we have imposed from the $'$ in the second summation. We required that $(W, d_i) = 1, (W, e_i) = 1$ and $(d_i, e_j) = 1$ for $i \neq j$. Note that the first two conditions can be removed due to the restriction of the support of $\lambda_{d_1,\ldots,d_k}$. We remove the last condition $(d_i, e_j) = 1$ by introducing variables $s_{i, j}$ for $i \neq j$. Since \[ \sum_{x \mid d_i, e_j} \mu(x) = 1 \] if $(d_i, e_j) = 1$ and the quantity is $0$ otherwise, the previous sum can be rewritten as \[ \sum_{\substack{u_1, \ldots, u_k \\ u_m = 1}} \prod_{i \neq m} g(u_i) \sum_{\substack{d_1, \ldots, d_k \\ e_1, \ldots, e_k \\ u_i \mid d_i, e_i \\ d_m, e_m \mid q}}\frac{\lambda_{d_1, \ldots, d_k}\lambda_{e_1, \ldots, e_k}}{\prod_{i \neq m} \phi(d_i)\phi(e_i)} \sum_{\substack{s_{1, 2}, \ldots, s_{k, k-1} \\ s_{i, j} \mid d_i, e_j}} \prod_{\substack{1 \le i, j \le k \\ i \neq j}} \mu(s_{i, j}). \] Interchanging the order of summation between the $s$ and $d$ variables gives us \[ \sum_{\substack{u_1, \ldots, u_k \\ u_m = 1}} \prod_{i \neq m} g(u_i) \sum_{\substack{s_{1, 2}, \ldots, s_{k, k-1}}} \prod_{\substack{1 \le i, j \le k \\ i \neq j}} \mu(s_{i, j}) \sum_{\substack{d_1, \ldots, d_k \\ e_1, \ldots, e_k \\ u_i \mid d_i, e_i \\ s_{i, j} \mid d_i, e_j \\ d_m, e_m \mid q}}\frac{\lambda_{d_1, \ldots, d_k}\lambda_{e_1, \ldots, e_k}}{\prod_{i \neq m} \phi(d_i)\phi(e_i)}. \] Note that in the previous sum, we can restrict to the terms where $(s_{i, a}, s_{i, j}) = 1$ for any $a \neq j$. This is because $(s_{i, a}, s_{i, j}) \mid (e_a, e_j)$, and $\lambda_{e_1, \ldots, e_k}$ is supported on $(e_1,\ldots,e_k)$ such that the $e_i$ are pairwise relatively prime. Similarly, we can assume that $(s_{a, j}, s_{i, j}) = 1, (s_{i, j}, u_i) = 1$, and $(s_{i, j}, u_j) = 1.$ We will denote these restrictions by $\sum^{\star}$. As we have separated the $d$ and $e$ variables completely, we can substitute in the $y^{(m, q)}$ variables to rewrite the previous expression as \[ \sum_{\substack{u_1, \ldots, u_k \\ u_m = 1}} \prod_{i \neq m} g(u_i) \sideset{}{^\star}\sum_{\substack{s_{1, 2}, \ldots, s_{k, k-1}}} \prod_{\substack{1 \le i, j \le k \\ i \neq j}} \mu(s_{i, j}) \left(\prod_{i = 1}^k \frac{\mu(a_i)\mu(b_i)}{g(a_i)g(b_i)} \right) y_{a_1, \ldots, a_k}^{(m, q)}y_{b_1, \ldots, b_k}^{(m, q)}, \] where 
\[
a_j = u_j \prod_{i\neq j}s_{j, i}\textnormal{ and } b_j = u_j \prod_{i\neq j} s_{i, j}.
\]
Now, we use the fact that $s_{i, j}$ is relatively prime with $s_{i, a}, s_{b, j}, u_i,$ and $u_j$ in order to write $\mu(a_j) = \mu(u_j)\prod_{i \neq j}\mu(s_{j, i})$, and similar expressions for $\mu(b_j), g(a_j),$ and $g(b_j).$ Substituting these into the previous expression, we have \[ \sum_{\substack{u_1, \ldots, u_k \\ u_m = 1}} \Big(\prod_{i \neq m} \frac{\mu(u_i)^2}{g(u_i)} \Big) \sideset{}{^\star}\sum_{\substack{s_{1, 2}, \ldots, s_{k, k-1}}} \Big(\prod_{\substack{1 \le i, j \le k \\ i \neq j \\ i, j \neq m}} \frac{\mu(s_{i, j})}{g(s_{i, j})^2}\Big) \Big(\prod_{\substack{i = m \text{ or } \\ j = m}} \frac{1}{g(s_{i, j})}\Big) y_{a_1, \ldots, a_k}^{(m, q)}y_{b_1, \ldots, b_k}^{(m, q)}. \] 

%triple divisibility
We now bound the contribution from the terms with $s_{i, j} > 1$ in the previous expression. Since $s_{i, j} \mid a_i$, we have $y^{(m, q)}_{a_1, \ldots, a_k} = 0$ for $1 < s_{i, j} \le D_0.$ Therefore, we may assume that $s_{i, j} = 1$ or $s_{i, j} > D_0.$ If $i = m$ or $j = m$, then $s_{i, j} = 1$ or $s_{i, j} \ge N^\eta$, as $s_{m, j}$ divides $d_m$, which divides $q$. Thus, $s_{m, j}$ would only have factors of size at least $N^\eta.$ First, let us consider the case where $s_{i, j} > D_0$, and $i \neq m, j \neq m.$ The contribution from these terms is bounded by \[ (y_{\max}^{(m, q)})^2 \Big(\sum_{\substack{(u, W) = 1 \\ u < R}} \frac{1}{g(u)}\Big)^{k-1}\Big(\sum_{s \ge 1} \frac{1}{g(s)^2} \Big)^{k^2-3k+1} \Big(\sum_{s_{i, j} > D_0} \frac{1}{g(s_{i, j})^2} \Big)\Big(\sum_{t \mid q} \frac{1}{g(t)}\Big)^{2k-2} \] \[ \ll \frac{(y_{\max}^{(m, q)})^2 \phi(W)^{k-1}(\log R)^{k-1}}{W^{k-1}D_0}. \] Here we have used the fact that $q$ has $2^{r-1}$ factors, and thus \[ \sum_{t \mid q} \frac{1}{g(t)} \le 2^{r-1}. \] The contribution from the terms with $s_{i, j} > 1$, where $i = m$ or $j = m$, is bounded by \[ (y_{\max}^{(m, q)})^2 \Big(\sum_{\substack{(u, W) = 1 \\ u < R}} \frac{1}{g(u)}\Big)^{k-1}\Big(\sum_{s \ge 1} \frac{1}{g(s)^2} \Big)^{k^2-3k+2} \Big(\sum_{t \mid q} \frac{1}{g(t)}\Big)^{2k-3} \Big(\sum_{\substack{s_{i, j} \mid q \\ s_{i, j} > 1}} \frac{1}{g(s_{i, j})} \Big) \] \[ \ll \frac{(y_{\max}^{(m, q)})^2 \phi(W)^{k-1}(\log R)^{k-1}}{W^{k-1}N^\eta}. \] We have additionally used the fact that \[ \Big(\sum_{\substack{s_{i, j} \mid q \\ s_{i, j} > 1}} \frac{1}{g(s_{i, j})} \Big) \le \frac{2^{r-1}}{N^\eta - 2}, \] as all prime factors of $q$ are at least $N^\eta.$ The first error term dominates. By simplifying the main term, where $s_{i, j} = 1$ for all $i, j$, we see that \[ \sideset{}{'}\sum_{\substack{d_1, \ldots, d_k \\ e_1, \ldots, e_k \\ d_m, e_m \mid q}}\frac{\lambda_{d_1, \ldots, d_k}\lambda_{e_1, \ldots, e_k}}{\prod_{i \neq m} \phi([d_i, e_i])} = \sum_{\substack{u_1, \ldots, u_k \\ u_m = 1}} \frac{(y_{u_1, \ldots, u_k}^{(m, q)})^2}{\prod_{i \neq m}g(r_i)} + O\Big(\frac{(y^{(m, q)}_{\max})^2\phi(W)^{k-1} (\log R)^{k-1}}{W^{k-1}D_0} \Big). \]
\end{proof}

In order to relate $S_2^{(m)}$ to $S_1$, we now express the $y^{(m, q)}$ variables in terms of the $y$ variables.

\begin{lemma}
\label{ymq}
If $r_m \mid q$, then \[ y_{r_1, \ldots, r_k}^{(m, q)} = \sum_{r_m \mid a_m} \frac{y_{r_1, \ldots, r_{m-1}, a_m, r_{m+1}, \ldots, r_k}}{\phi(a_m)}\mu((a_m, q))\phi((a_m, q)) + O\Big(\frac{y_{\max}\phi(W)\log R}{WD_0} \Big). \] In particular, for $r_m = 1$, the previous expression can be written as \[ y_{r_1, \ldots, r_k}^{(m, q)} = \sum_{d \mid q} \mu(d)d \sum_{d \mid a_m} \frac{y_{r_1, r_{m-1}, a_m, r_{m+1}, \ldots, r_k}}{\phi(a_m)} + O\Big(\frac{y_{\max}\phi(W)\log R}{WD_0} \Big).\]
\end{lemma}

\begin{proof}
As stated in \cite[Equation 5.8]{maynard}, we have the inversion formula $\lambda_{d_1,\ldots,d_k}$ in terms of the variables $y_{r_1,\ldots,r_k}$: \[ \lambda_{d_1, \ldots, d_k} = \prod_{i = 1}^m \mu(d_i)d_i \sum_{\substack{r_1, \ldots, r_k \\ d_i \mid r_i}} \frac{y_{r_1, \ldots, r_k}}{\prod_{i = 1}^m \phi(r_i)}. \] Substituting this expression into \eqref{yeezy} gives us \[ y_{r_1, \ldots, r_k}^{(m, q)} = \prod_{i \neq m} \mu(r_i)g(r_i) \sum_{\substack{r_i \mid d_i \\ d_m \mid q}} \frac{\prod_{i = 1}^m\mu(d_i)d_i}{\prod_{i \neq m} \phi(d_i)} \sum_{d_i \mid a_i} \frac{y_{a_1, \ldots, a_k}}{\prod_{i = 1}^m \phi(a_i)}. \] We now swap the order of summation between the $d$ and $a$ variables to give \[ y_{r_1, \ldots, r_k}^{(m, q)} = \prod_{i \neq m} \mu(r_i)g(r_i) \sum_{r_i \mid a_i}\frac{y_{a_1, \ldots, a_k}}{\prod_{i = 1}^m \phi(a_i)} \Big(\prod_{i \neq m} \sum_{\substack{d_i \mid a_i \\ r_i \mid d_i}} \frac{\mu(d_i)d_i}{\phi(d_i)}\Big)\Big(\sum_{\substack{d_m \mid (a_m, q) \\ r_m \mid d_m}} \mu(d_i)d_i\Big).\] We can explicitly compute the innermost sum: for $i \neq m$, we have \[ \sum_{\substack{d_i \mid a_i \\ r_i \mid d_i}} \frac{\mu(d_i)d_i}{\phi(d_i)} = \frac{\mu(a_i)r_i}{\phi(a_i)}, \] and for $i = m$, we have \[ \sum_{\substack{d_m \mid (a_m, q) \\ r_m \mid d_m}} \mu(d_m)d_m = \mu((a_m, q))\phi((a_m, q))\frac{r_m}{\phi(r_m)}. \]
Substituting these into the previous expression gives us \begin{align*} 
y_{r_1, \ldots, r_k}^{(m, q)} &= \frac{r_m}{\phi(r_m)}\prod_{i \neq m}\mu(r_i)g(r_i)\sum_{r_i \mid a_i} \frac{y_{a_1, \ldots, a_k}}{\prod_{i=1}^m \phi(a_i)} \Big(\prod_{i \neq m} \frac{\mu(a_i)r_i}{\phi(a_i)}\Big) \mu((a_m, q))\phi((a_m, q)) \\ &= \frac{r_m}{\phi(r_m)}\prod_{i \neq m} \mu(r_i)g(r_i)r_i \sum_{r_i \mid a_i} \frac{y_{a_1, \ldots, a_k}}{\phi(a_m)}\mu((a_m, q))\phi((a_m, q)) \prod_{i \neq m} \frac{\mu(a_i)}{\phi(a_i)^2}
\end{align*} We now claim that for $i \neq m$, the contribution from the terms with $a_i \neq r_i$ is small. Note that we must have $a_i > D_0r_i$, as $y_{a_1, \ldots, a_k}$ is only supported when $(a_i, W) > 1.$ Suppose that $a_j > D_0r_j.$ The contribution from these terms is \begin{align*} &\ll \frac{y_{\max}r_m}{\phi(r_m)} \prod_{i \neq m} g(r_i)r_i \Big(\sum_{\substack{r_j \mid a_j \\ a_j > D_0r_j}} \frac{1}{\phi(a_i)^2} \Big) \prod_{\substack{i \neq m, j}} \Big(\sum_{r_i \mid a_i} \frac{1}{\phi(a_i)^2} \Big) \sum_{\substack{r_m \mid a_m \\ (a_m, W) = 1 \\ a_m < R}} \frac{\phi((a_m, q))}{\phi(a_m)} \\ &\ll \frac{y_{\max}r_m}{D_0\phi(r_m)} \prod_{i \neq m} \Big(\frac{g(r_i)r_i}{\phi(r_i)^2}\Big) \Big(\sum_{d \mid q}\sum_{\substack{d \mid a_m \\ (a_m, W) = 1 \\ a_m < R}} \frac{\phi(d)}{\phi(a_m)}\Big) \\ &\ll \frac{y_{\max}r_m}{D_0\phi(r_m)} \sum_{d \mid q} \frac{\phi(W) \log R}{W} \ll \frac{y_{\max} \phi(W) \log R}{WD_0}. \end{align*}
The last inequality holds because $q$ has at most $2^r$ factors and $r_m \mid q$, and thus \[ 1 \le \frac{r_m}{\phi(r_m)} \le \Big(\frac{N^\eta}{N^\eta-1}\Big)^r = 1 + O(N^{-\eta}), \]
Therefore, \[ y_{r_1, \ldots, r_k}^{(m, q)} = \frac{r_m}{\phi(r_m)}\prod_{i \neq m} \frac{g(r_i)r_i}{\phi(r_i)^2} \sum_{r_m \mid a_m} \frac{y_{r_1, \ldots, r_{m-1}, a_m, r_{m+1}, \ldots, r_k}}{\phi(a_m)}\mu((a_m, q))\phi((a_m, q)) + O\Big(\frac{y_{\max}\phi(W) \log R}{WD_0} \Big). \]
We conclude that \[ \frac{r_m}{\phi(r_m)}\prod_{i \neq m} \frac{g(r_i)r_i}{\phi(r_i)^2} \sum_{r_m \mid a_m} \frac{y_{r_1, \ldots, r_{m-1}, a_m, r_{m+1}, \ldots, r_k}}{\phi(a_m)}\mu((a_m, q))\phi((a_m, q)) \] \[ = \prod_{i \neq m} \frac{g(r_i)r_i}{\phi(r_i)^2} \sum_{r_m \mid a_m} \frac{y_{r_1, \ldots, r_{m-1}, a_m, r_{m+1}, \ldots, r_k}}{\phi(a_m)}\mu((a_m, q))\phi((a_m, q)) + O\Big(\frac{y_{\max}\phi(W) \log R}{WN^\eta} \Big), \] where the bounding is done analogously as above. To finish the proof, note that $g(\ell)\ell/\phi(\ell)^2 = 1 + O(\ell^{-2}).$ This gives that \[\prod_{i \neq m} \frac{g(r_i)r_i}{\phi(r_i)^2} = 1 + O(D_0^{-1}). \] Thus, our final expression is \[ y_{r_1, \ldots, r_k}^{(m, q)} = \sum_{r_m \mid a_m} \frac{y_{r_1, \ldots, r_{m-1}, a_m, r_{m+1}, \ldots, r_k}}{\phi(a_m)} \mu((a_m, q))\phi((a_m, q)) + O\Big(\frac{y_{\max} \phi(W)\log R}{WD_0}\Big), \] which completes the proof.
\end{proof}

\section{Smooth Choice of $y$}\label{sec:smoothy}
When $\mu^2(W\prod_{i = 1}^k r_i) = 0$, we have $y_{r_1, \ldots, r_k} = 0$. Otherwise, we will set \[ y_{r_1, \ldots, r_k} = F\Big(\frac{\log r_1}{\log R}, \ldots, \frac{\log r_k}{\log R}\Big) \] for some smooth function $F$ to be chosen in Section~\ref{sec:smoothf}. Note that this implies that $y_{\max} \le F_{\max}$. Additionally, by Lemma \ref{ymq}, we can show that \begin{align} y_{r_1, \ldots, r_k}^{(m, q)} &\ll y_{\max} \sum_{r_m \mid a_m} \frac{\phi((a_m, q))}{\phi(a_m)} \nonumber \\ &\ll y_{\max} \sum_{d \mid q} \sum_{\substack{d \mid a_m \\ (a_m, W) = 1 \\ a_m < R}} \frac{\phi(d)}{\phi(a_m)} \nonumber \\ &\ll \frac{y_{\max}\phi(W)\log R}{W} \ll \frac{F_{\max}\phi(W)\log R}{W}, \end{align} where the bounding is done as in Lemma \ref{ymq}.

We will use the following lemma to estimate the sum $S_2^{(m)}$ by an integral.
\begin{lemma}
\label{gpylemma}
Let $A_1, A_2, L > 0$, and let $\gamma$ be a multiplicative function satisfying \[ 0 \le \frac{\gamma(\ell)}{\ell} \le 1-A_1, \] \[ -L \le \sum_{w \le \ell \le z} \frac{\gamma(\ell)\log \ell}{\ell} - \log \frac{z}{w} \le A_2 \] for all $2 \le w \le z.$ Let $h$ be the totally multiplicative function defined on the primes by $h(\ell) = \gamma(\ell)/(\ell-\gamma(\ell)).$ Let $G$ be a smooth function from $[0, 1] \rightarrow \mathbb{R}$, and let $G_{\max} = \max_{t \in [0, 1]} |G(t)|+|G'(t)|.$ Then for any positive integer $n$, \[ \sum_{j < R/n} \mu^2(j)h(j)G\Big(\frac{\log jn}{\log R} \Big) = \mathfrak{S} \log \frac{R}{n} \int_0^1 G\Big(\frac{\log(R/n)}{\log R} \Big(\frac{\log n}{\log (R/n)} + x\Big)\Big) dx + O(\mathfrak{S}LG_{\max}), \] where the implied constant does not depend on $G$ or $L$, and \[ \mathfrak{S} \defeq \prod_\ell \Big(1 - \frac{\gamma(\ell)}{\ell} \Big)^{-1} \Big(1 - \frac{1}{\ell} \Big). \]
\end{lemma}
\begin{proof}
This is stated in \cite[Lemma 3.5]{sono} as a direct consequence of \cite[Lemma 4]{gpylemma}.
\end{proof}
Our main claim of this section is expressed in the following lemma.

\begin{lemma} 
\label{s2mint}
We have \begin{align} S_2^{(m)} &= \frac{\phi(B)\phi(W)^k (\log R)^{k-1}}{W^{k+1}} \nonumber\\ \label{needonce} &\times \int_0^1 \ldots \int_0^1 \Big[\sideset{}{^{\dagger}}\sum_q X_{N/q} \Big( \sum_{d \mid q} \mu(d) \log \frac{R}{d} \int_0^1 G^{(m)}_d(\vec{t_i}) dt_m \Big)^2 \Big] dt_1\ldots dt_{m-1}dt_{m+1} \ldots dt_k \\ &+ O\Big(\frac{F_{\max}^2 \mathcal{T}_N\phi(B)\phi(W)^k (\log R)^{k+1}}{W^{k+1}D_0} \Big),\nonumber \end{align}
where \[ G^{(m)}_d(\vec{t_i}) := F\Big(t_1, \ldots, t_{m-1}, \frac{\log(R/d)}{\log R} \Big(\frac{\log d}{\log R/d}+t_m\Big), t_{m+1}, \ldots, t_k\Big) \]
 and \[ F_{\max} \defeq \sup_{(x_1, \ldots, x_k) \in [0, 1]^k} |F(x_1, \ldots, x_k)| + \sum_{i = 1}^k \Big|\frac{\partial F}{\partial x_i}(x_1, \ldots, x_k) \Big|. \]
\end{lemma}

Before proving Lemma \ref{s2mint}, we first prove the following claim, which expresses the $y^{(m, q)}$ variables in terms of the function $F$.

\begin{lemma}
\label{ymqint}
If $r_m = 1$, and $y_{r_1, \ldots, r_k}^{(m, q)} \neq 0$, then
\[ y_{r_1, \ldots, r_k}^{(m, q)} = \frac{\phi(W)}{W} \prod_{i \neq m} \frac{\phi(r_i)}{r_i}\sum_{d \mid q} \mu(d) \log\frac{R}{d} \int_0^1 H_d(t_m) dt_m + O\Big(\frac{F_{\max}\phi(W)\log R}{WD_0} \Big), \]
where \[ H_d(t_m) = F\Big(\frac{\log r_1}{\log R}, \ldots, \frac{\log r_{m-1}}{\log R}, \frac{\log(R/d)}{\log R} \Big(\frac{\log d}{\log R/d}+t_m\Big), \frac{\log r_{m+1}}{\log R}, \ldots, \frac{\log r_k}{\log R} \Big). \]
\end{lemma}

\begin{proof}
We apply Lemma \ref{gpylemma} to the expression for $y_{r_1, \ldots, r_k}^{(m, q)}$ described at the end of Lemma \ref{ymq}. By Lemma \ref{ymq}, for $r_m = 1$, we can write
{\footnotesize
\begin{align*} 
y_{r_1, \ldots, r_k}^{(m, q)} &= \sum_{d \mid q} \mu(d)d \sum_{d \mid a_m} \frac{y_{r_1, \ldots, r_{m-1}, a_m, r_{m+1}, \ldots, r_k}}{\phi(a_m)} + O\left(\frac{y_{\max}\phi(W)\log R}{WD_0} \right) \\ &= \sum_{d \mid q} \frac{\mu(d)d}{\phi(d)} \sum_{\substack{a_m' < R/d \\ (a_m', dW\prod_{i\neq m} r_i) = 1}} F\left(\frac{\log r_1}{\log R}, \ldots, \frac{\log r_{m-1}}{\log R}, \frac{\log (a_m' \cdot d)}{\log R}, \frac{\log r_{m+1}}{\log R}, \ldots, \frac{\log a_k}{\log R} \right) \cdot \frac{\mu(a_m')^2}{\phi(a_m')} \\ &\hspace{24pt}+ O\left(\frac{y_{\max}\phi(W)\log R}{WD_0}\right). 
\end{align*}} We choose $\gamma$ on the primes in the following way: \[ \gamma(\ell) = \begin{cases} 0 \textnormal{ if } \ell \mid dW\prod_{i \neq m}r_i \\ 1 \textnormal{ otherwise} \\ \end{cases}. \] Using the notation in Lemma \ref{gpylemma}, we can use the same argument as \cite{sono} and choose \begin{align*} L &\ll 1+ \sum_{\ell \mid Wd\prod_{i \neq m} r_i} \frac{\log \ell}{\ell} + \left(\log{z/w}-\sum_{w \le \ell \le z} \frac{\log \ell}{\ell} \right) \\ &\ll 1 + \sum_{\ell < \log R} \frac{\log \ell}{\ell} + \sum_{\substack{\ell > \log R \\ \ell \mid Wd\prod_{i\neq m} r_i}} \frac{\log \log R}{\log R} \ll \log \log N \end{align*} where $A_1$ and $A_2$ are fixed constants.

Note that because $y_{r_1, \ldots, r_k}^{(m, q)}$ is supported, $(r_i, r_j) = 1$ and $(r_i, W) = 1.$ If $(d, W) > 1$ or $(d, r_i) > 1$ for some $i \neq m$, then clearly $y_{r_1, r_{m-1}, a_m, r_{m+1}, \ldots, r_k} = 0$, so we can ignore these terms. Therefore, \[ \mathfrak{S} = \frac{\phi(dW\prod_{i\neq m} r_i)}{dW \prod_{i\neq m} r_i} = \frac{\phi(d)\phi(W)\prod_{i\neq m} \phi(r_i)}{dW\prod_{i\neq m} r_i}. \] Directly substituting this into Lemma \ref{gpylemma} yields \begin{align*} y_{r_1, \ldots, r_k}^{(m, q)} &= \frac{\phi(W)}{W} \prod_{i \neq m} \frac{\phi(r_i)}{r_i}\sum_{d \mid q} \mu(d) \log\frac{R}{d} \int_0^1 H_d(t_m) dt_m \\ &+ \sum_{d \mid q} \frac{d}{\phi(d)} O\left(\frac{F_{\max}\log \log N\phi(d)\phi(W) \prod_{i\neq m} \phi(r_i)}{dW\prod_{i\neq m} r_i} \right) + O\left(\frac{F_{\max}\phi(W)\log R}{WD_0} \right). \end{align*} The second error term dominates, which completes the proof.
\end{proof}

We are now ready to prove Lemma \ref{s2mint}.

\begin{proof}
For simplicity, define \[ \mathcal{G}^{(q)}_{r_1, \ldots, r_{m-1}, r_{m+1}, \ldots, r_k} = \sum_{d \mid q} \mu(d) \log\frac{R}{d} \int_0^1 H_d(t_m) dt_m, \] where $H_d$ is defined as in Lemma \ref{ymqint}. Note that \[ \sup_{(r_1, \ldots, r_{m-1}, r_{m+1}, \ldots, r_k) \in \{1, 2, \ldots, R\}^{k-1}} \mathcal{G}^{(q)}_{r_1, \ldots, r_{m-1}, r_{m+1}, \ldots, r_k} \ll F_{\max} \log R. \] Now, note that by Lemma \ref{ymqint}, for $r_m = 1$ and $y_{r_1, \ldots, r_k}^{(m, q)} \neq 0$, we have \[ (y_{r_1, \ldots, r_k}^{(m, q)})^2 = \frac{\phi(W)^2 \prod_{i \neq m} \phi(r_i)^2}{W^2 \prod_{i \neq m} r_i^2} (\mathcal{G}^{(q)}_{r_1, \ldots, r_{m-1}, r_{m+1}, \dots, r_k})^2 + O\left(\frac{F_{\max}^2\phi(W)^2(\log R)^2}{W^2D_0} \right). \] Therefore, we can write \begin{equation} \label{kanye} S_2^{(m)} = \frac{\phi(B)\phi(W)}{W^2} \sum_{\substack{r_i, i \neq m \\ (r_i, W) = 1 \\ (r_i, r_j) = 1}} \prod_{i \neq m} \left(\frac{\phi(r_i)^2\mu(r_i)^2}{g(r_i)r_i^2} \right) \sideset{}{^\dagger} \sum_q X_{N/q} (\mathcal{G}^{(q)}_{r_1, \ldots, r_{m-1}, r_{m+1}, \dots, r_k})^2, \end{equation} plus an error term. This error is \[ \ll \frac{F_{\max}^2\mathcal{T}_N\phi(B)\phi(W)(\log R)^2}{W^2D_0} \left(\sum_{\substack{r < R \\ (r, W) = 1}} \frac{1}{g(r)} \right)^{k-1} \ll \frac{F_{\max}^2\mathcal{T}_N\phi(B)\phi(W)^k (\log R)^{k+1}}{W^{k+1}D_0}, \] which is negligible. Looking at the main term of $S_2^{(m)}$ in \eqref{kanye}, we will now eliminate the condition that $(r_i, r_j) = 1$ from the sum, at the cost of a small error. Note that if $(r_i, r_j) > 1$, then $(r_i, r_j) > D_0.$ Therefore, by summing over all primes $p$ greater than $D_0$, we observe that the contribution from terms with $(r_i, r_j) > 1$ is 
\begin{align*}
\ll \frac{F_{\max}^2\mathcal{T}_N\phi(B)\phi(W)(\log R)^2}{W^2}& \left(\sum_{p > D_0} \frac{\phi(p)^4}{g(p)^2p^4} \right) \left(\sum_{\substack{r < R \\ (r, W) = 1}} \frac{\phi(r)^2}{g(r)r^2} \right)^{k-1} \\&\ll \frac{F_{\max}^2\mathcal{T}_N\phi(B)\phi(W)^k (\log R)^{k+1}}{W^{k+1}D_0}, 
\end{align*}
which is also negligible. Therefore, \begin{align} S_2^{(m)} &= \frac{\phi(B)\phi(W)}{W^2} \sum_{\substack{r_i, i \neq m \\ (r_i, W) = 1}} \prod_{i \neq m} \left(\frac{\phi(r_i)^2\mu(r_i)^2}{g(r_i)r_i^2} \right) \sum_{q} (\mathcal{G}^{(q)}_{r_1, \ldots, r_{m-1}, r_{m+1}, \dots, r_k})^2 \\ &+ O\left(\frac{F_{\max}^2\phi(B)\phi(W)^k (\log R)^{k+1}}{W^{k+1}D_0}\right). \nonumber \end{align}

Now, consider $\gamma$ defined on the primes in the following way: \[ \gamma(\ell) = \begin{cases} 0 \textnormal{ if }\ell \mid W \\ 1 - \frac{\ell^2-3\ell+1}{\ell^3-\ell^2-2\ell+1} \textnormal{ otherwise.} \\ \end{cases} \] This was chosen so that $h(\ell) = 0$ if $\ell \mid W$, and $h(\ell) = \frac{\phi(\ell)^2}{g(\ell)\ell^2}$ otherwise. As we want to apply Lemma \ref{gpylemma}, we can choose \[ L \ll 1 + \prod_{\ell \mid W} \frac{\log \ell}{\ell} \ll \log D_0. \] Note that \begin{align*} \mathfrak{S} &= \prod_\ell \left(1 - \frac{\gamma(\ell)}{\ell} \right)^{-1} \left(1 - \frac{1}{\ell} \right) \\ &= \frac{\phi(W)}{W} \prod_{\ell \nmid W} \left(1 - \frac{\gamma(\ell)}{\ell} \right)^{-1} \left(1 - \frac{1}{\ell} \right) = \frac{\phi(W)}{W} \prod_{\ell \nmid W} \left(1 - O(\ell^{-2}) \right) = \frac{\phi(W)}{W}(1 - O(D_0^{-1})). \end{align*} Applying Lemma \ref{gpylemma} $k-1$ times and substituting in the definition of $\mathcal{G}^{(q)}_{r_1, \ldots, r_{m-1}, r_{m+1}, \ldots, r_k}$ and $H_d(t_m)$, we have that \begin{align*} S_2^{(m)} &= \frac{\phi(B)\phi(W)^k (\log R)^{k-1}}{W^{k+1}} \\ &\times \int_0^1 \ldots \int_0^1 \left[\sideset{}{^{\dagger}}\sum_q X_{N/q} \left( \sum_{d \mid q} \mu(d) \log \frac{R}{d} \int_0^1 G^{(m)}_d(\vec{t_i}) dt_m \right)^2 \right] dt_1\ldots dt_{m-1}dt_{m+1} \ldots dt_k \\ &+ O\left(\frac{F_{\max}^2 \mathcal{T}_N \phi(B)\phi(W)^k (\log R)^k \log D_0}{W^{k+1}D_0} \right) + O\left(\frac{F_{\max}^2 \mathcal{T}_N\phi(B)\phi(W)^k (\log R)^{k+1}}{W^{k+1}D_0} \right), \end{align*} where the first error term comes from Lemma \ref{gpylemma} and the second comes from earlier. The first error term is negligible, which completes the proof.
\end{proof}

As each $q$ in $\sum^{\dagger}$ only has prime factors that are at least $N^\eta$, we know that every divisor $d > 1$ of $q$ is at least $N^\eta.$ Therefore, if $F$ is chosen to be supported on $x_i \le \frac{\log N^{\eta}}{\log R}$, then all terms with $d > 1$ in (\ref{needonce}) vanish. Since $\frac{\log N^{\eta}}{\log R} \ge \frac{2\eta}{\theta}$, we observe that imposing $F$ to be supported on
\begin{equation}\label{setsupport}
\left\{(x_1,\ldots, x_k) \in [0,1]^k \colon x_i \le \frac{2\eta}{\theta}\right\}
\end{equation}
accomplishes this goal. Given this discussion, we have the following corollary.
\begin{corollary}\label{coral}
If $F$ is supported on \eqref{setsupport}, then \[ S_2^{(m)} = (1 + o(1))\frac{\mathcal{T}_N\phi(B)\phi(W)^k(\log R)^{k + 1}}{W^{k + 1}}J_k^{(m)}(F). \]
\end{corollary}

\section{Smooth Choice of $F$}\label{sec:smoothf}
We now choose a constant $\eta$ and a function $F$ to give a large lower bound for \eqref{functional}.
\begin{lemma}
\label{calcvar}
Define \[ I_k(F) = \int_0^1 \int_0^1 \ldots \int_0^1 F(t_1, \ldots, t_k)^2 dt_1 \ldots dt_k \] and \[ J_k^{(m)}(F) = \int_0^1 \int_0^1 \ldots \int_0^1 \left(\int_0^1 F(t_1, \ldots, t_k) dt_m \right)^2 dt_1 \ldots dt_{m-1} dt_{m+1} \ldots dt_k. \] Let $A >0$,
\begin{equation}\label{notationof}
T = \frac{e^A-1}{A}, \hspace{18pt}\text{ and }\hspace{18pt}\sigma = \frac{1}{A^2}(A-1+e^{-A}). 
\end{equation} If $1 - \frac{T}{k} - \sigma > 0$, then there exists a smooth symmetric function $F: [0,1]^k \to \RR$ supported on $\{(x_1, \ldots, x_k) \colon x_i \le T/k \}$ such that \[ \frac{\sum_{m=1}^k J_k^{(m)}(F)}{I_k(F)} \ge A\left(1 - \frac{Ae^A}{k(1 - A/(e^A-1) - e^A/k)^2} \right). \]
\end{lemma}

\begin{proof}
This is proved in \cite[Section 7]{maynard} for 
\[
F(x_1,\ldots,x_k)\defeq \begin{cases} \prod_{i=1}^k g(kx_i) &\mbox{if } \sum_{i=1}^k x_i \le 1, \\ 
0 & \mbox{otherwise}, \end{cases}
\]
where 
\[
g(t)\defeq \begin{cases} \frac{1}{1+At} &\mbox{if } 0 \le t \le T, \\ 
0 & \mbox{otherwise}. \end{cases} 
\]
\end{proof}
The function $F$ defined above is supported on $\{(x_1, \ldots, x_k) \colon x_i \le T/k \}$. Since we need $F$ to be supported on \eqref{setsupport} in order to apply Corollary~\ref{coral}, we let
\[
\eta = \frac{T\theta}{2k},
\]
Setting $y_{r_1,\ldots,r_k}$ equal to $F(\frac{\log r_1}{\log R},\ldots,\frac{\log r_k}{\log R})$ for the function $F$ given in Lemma \ref{s2mint}, we have \[ S_2^{(m)} = (1+o(1))\frac{\mathcal{T}_N\phi(B)\phi(W)^k (\log R)^{k+1}J_k^{(m)}(F)}{W^{k+1}}, \] 
Also, recall from \eqref{s1s1} that 
\[
S_1 = (1+o(1)) \frac{NB \phi(W)^k (\log R)^k I_k(F)}{W^{k+1}}. \] 
We now prove the following lower bound on $\mathcal{T}_N$ stated in Proposition~\ref{prop1}.
\begin{lemma}
For $0 < \eta < \frac{1}{2(r-1)}$, we have
\[
\mathcal{T}_N \ge (1+o(1))\frac{\delta^r N}{(\log N)(r-1)!}\left(\log \frac{1}{2(r-1)\eta} \right)^{r-1}.
\]
\end{lemma}
\begin{proof}
We have \[ \mathcal{T}_N \ge (1+o(1))\sum_{\substack{q = p_1\ldots p_{r-1} \\ N^\eta \le p_i \le N^{\frac{1}{2(r-1)}} \\ p_1 < \ldots < p_{r-1}}} \frac{\delta N}{\log(N/q)} \frac{1}{q}, \] where the $1+o(1)$ factor comes from \ref{2.1}. Simplifying the above expression, we have that $ \mathcal{T}_N$ is at least \begin{align*} & (1+o(1)) \frac{\delta N}{\log N} \sum_{\substack{q = p_1\ldots p_{r-1} \\ N^\eta \le p_i \le N^\frac{1}{2(r-1)} \\ p_1 < \ldots < p_{r-1}}} \frac{1}{q} = (1+o(1))\frac{\delta N}{(\log N) (r-1)!} \left(\sum_{\substack{N^\eta \le p \le N^\frac{1}{2(r-1)} \\ p \in \mathcal{P}}} \frac{1}{p} \right)^{r-1} + E_{\textup{rep}}, \end{align*} where $E_{\textup{rep}}$ denotes the contribution from repeated primes. This contribution is \[  \ll \frac{\delta N}{\log N} \left(\sum_{N^\eta \le p \le N^\frac{1}{2(r-1)}} \frac{1}{p^2} \right) \left( \sum_{N^\eta \le p \le N^\frac{1}{2(r-1)}} \frac{1}{p} \right)^{\min(0,r-3)} \ll \frac{N (\log N)^r}{N^\eta}, \] which is negligible.

Let $\eps = \frac{1}{\log N}$, and split the interval $[N^\eta, N^\frac{1}{2(r-1)}]$ into intervals $[(1+\eps)^jN^\eta, (1+\eps)^{j+1} N^\eta]$ for all \[ 0 \le j \le \left(\log\left(\frac{N^\frac{1}{2(r-1)}}{N^\eta}\right) / \log (1 + \eps)-1\right). \] By \ref{2.1}, we have \begin{align*} \sum_{\substack{(1+\eps)^jN^\eta \le p \le (1+\eps)^{j+1}N^\eta \\ p \in \mathcal{P}}} \frac{1}{p} &\ge (1+o(1)) \frac{\delta((1+\eps)^{j+1} N^\eta - (1+\eps)^j N^\eta)}{(1+\eps)^{j+1} N^\eta\log((1+\eps)^{j+1} N^\eta)} \\ &= (1+o(1))\frac{\delta\eps}{(1+\eps)\log((1+\eps)^{j+1}N^\eta)} = (1+o(1))\frac{\delta\eps}{\log((1+\eps)^{j+1}N^\eta)}. \end{align*}
Finally, a direct computation shows that \[ \sum_{0 \le j \le \log\left(N^\frac{1}{2(r-1)}/N^\eta\right) / \log (1 + \eps) -1} \frac{\delta\eps}{\log((1+\eps)^{j+1}N^\eta)} = (1+o(1))\delta\log\left(\frac{1}{2(r-1)\eta}\right). \]
Overall, we have \[ \mathcal{T}_N \ge (1+o(1))\frac{\delta^r N}{(\log N)(r-1)!}\left(\log \frac{1}{2(r-1)\eta} \right)^{r-1}. \]
\end{proof}

We assume from now on that $k \ge e^r$, and choose $A = \frac{\log k}{r}$. Using the notation of \eqref{notationof}, one can check that $1 - \frac{T}{k} - \sigma > 0$, so it follows  by Lemma \ref{calcvar}  that 
\begin{equation}\label{final1}
\frac{\sum_{m=1}^k J_k^{(m)}(F)}{I_k(F)} \ge A\left(1 - \frac{Ae^A}{k(1 - A/(e^A-1) - e^A/k)^2} \right) \ge \frac{\log k}{r} - 1,
\end{equation}
since  $k \ge e^r$. Also, since $\eta = \frac{T\theta}{2k}$, we have \begin{equation}\label{final2}\log \frac{1}{2(r-1)\eta} \ge \log \frac{\frac{1}{r} (\log k) k^\frac{r-1}{r}}{(r-1) \theta} \ge \log \frac{k^\frac{r-1}{r}}{(r-1)\theta}, 
\end{equation}
where we have again used the fact that $k \ge e^r$. 

We have shown in Proposition~\ref{prop} that in order to complete the proof of Theorem \ref{thm:main}, it suffices to show that for some $k$, we have 
\[ \frac{\theta \phi(B)\delta^r \sum_{m=1}^k J^{(m)}_k(F)}{2B (r-1)! \cdot I_k(F)} \left(\log \frac{1}{2(r-1)\eta} \right)^{r-1} - \rho > 0. \]
By \eqref{final1} and \eqref{final2}, the above inequality follows from 
\begin{equation}\label{final3}  \frac{\log k/e^r}{r}  \left(\log \frac{k^\frac{r-1}{r}}{(r-1)\theta} \right)^{r-1} > \frac{2B\rho(r-1)!}{\phi(B) \delta^r \theta}. \end{equation} 
 Since $e^r \ge \left(\left(r-1\right)\theta\right)^\frac{r}{r-1}$ for $r \ge 2$ and $\theta \le 1$, it is easy to check that \eqref{final3} holds for 
\begin{equation}\label{final4} k >  \exp\left(r+\frac{r}{\delta}\left(\frac{2B\rho(r-1)!}{\phi(B)\theta(r-1)^{r-1}}\right)^\frac{1}{r}\right). \end{equation}
By our discussion in Subsection~\ref{outline}, this is sufficient to prove Theorem~\ref{thm:main}.

\section{An Example}\label{sec:example}

In this section, we compute a bound for $C(m)$ such that there are infinitely many $m$-tuples of $E_2$ numbers $d_1 \le \cdots \le d_m$ such that the following hold:
\begin{enumerate}
\item The class group $\Cl\left(\QQ(\sqrt{-d_i})\right)$ has an element of order $4$ for all $1 \le i \le m$.
\item  $d_m-d_1< C(m)$.
\end{enumerate}

We take $\mathcal{P} = \{p \in \mathbb{P}: p \equiv 1 \pmod 8\}$. Then, the set $\mathcal{P}$ has density $\delta = \frac{1}{4}$, satisfies a Siegel-Walfisz condition with $B = 2$, and has a level of distribution $\theta$ for any $\theta < \frac{1}{2}$. By \eqref{final4}, we have that \eqref{final3} holds for $\rho = m-1$ and $k > \exp(2 + 16\sqrt{2(m-1)})$. Setting $k = \lceil \exp(2 + 16\sqrt{2(m-1)}) \rceil$, we can choose $C(m)$ to be the diameter of any admissible set of size $k$ by Theorem \ref{thm:main}. We choose the admissible set to consist of the elements $n + h_1, n+h_2, \dots, n+h_k$, where $h_i$ is the $i$-th smallest prime greater than $k$. Using the bounds provided in \cite{dusart} as well as $k > \exp(18)$, we have that there are at most $\frac{k}{\log k} + \frac{k}{(\log k)^2}$ primes less than $k$, and that the $\lfloor k + \frac{k}{\l	og k} + \frac{k}{(\log k)^2} \rfloor$-th prime is at most
\[
k\left(1 + \frac{1}{\log k} + \frac{1}{(\log k)^2}\right)(\log k + \log\log k - 0.9061)
\]
Thus, we may choose
\[
C(m) = k\left(1 + \frac{1}{\log k} + \frac{1}{(\log k)^2}\right)(\log k + \log\log k - 0.9061) - k
\]
where $k = \lceil \exp(2 + 16\sqrt{2(m-1)}) \rceil$.

\section*{Acknowledgments}

\noindent This research was supervised by Ken Ono at the Emory University Mathematics REU and was supported by the National Science Foundation (grant number DMS-1557960). We would like to thank Ken Ono and Jesse Thorner for offering their advice and guidance and for providing many helpful discussions and valuable suggestions on the paper. We would also like to thank the anonymous referees for their helpful comments and suggestions.

\bibliographystyle{amsxport}
\bibliography{main.bbl}

\end{document}